\documentclass[12pt]{amsart}
\usepackage{nicematrix}
\usepackage[utf8]{inputenc}
\usepackage[T1]{fontenc}
\usepackage[english]{babel}
\usepackage{tikz}
\usepackage{tikz-cd}
\usepackage{mathscinet}
\usepackage{url}
\usepackage{multicol}
\usepackage{wrapfig}
\usepackage{float}
\usepackage[margin = 1in]{geometry}
\usepackage{amsmath}
\usepackage{amssymb}
\usepackage{amsfonts}
\usepackage{amsthm}
\usepackage{mathtools}
\usepackage{enumerate}
\usepackage{verbatim}
\usepackage{physics}
\usepackage{euscript}
\usepackage{tabularx}
\usepackage[normalem]{ulem}
\DeclareMathOperator{\id}{\mathrm{id}}

\newcommand{\CC}{\mathbb{C}}
\newcommand{\FF}{\mathbb{F}}
\newcommand{\Fpbar}{\overline{\mathbb{F}}_p}
\newcommand{\QQ}{\mathbb{Q}}
\newcommand{\ZZ}{\mathbb{Z}}

\DeclareMathOperator{\SL}{SL}
\DeclareMathOperator{\Frob}{Frob}

\DeclareMathOperator{\Aut}{Aut}
\DeclareMathOperator{\End}{End}
\DeclareMathOperator{\Pic}{Pic}
\DeclareMathOperator{\Cot}{Cot}
\DeclareMathOperator{\Id}{Id}

\newcommand{\OO}{\mathcal{O}}

\newcommand{\PP}{\mathbb{P}}

\newcommand{\Mod}[1]{\ (\mathrm{mod}\ #1)}

\DeclareMathOperator{\Cl}{Cl}

\DeclareMathOperator{\Hom}{Hom}

\DeclareMathOperator{\GL}{GL}
\DeclareMathOperator{\PSL}{PSL}

\DeclareMathOperator{\disc}{disc}

\DeclareMathOperator{\vertices}{vert}
\DeclareMathOperator{\edges}{edge}
\DeclareMathOperator{\realization}{real}
\DeclareMathOperator{\Stab}{stab}

\newcommand\legendre[2]{{
\left(\frac{#1}{#2}\right)
}}
\newcommand\restr[2]{{
  \left.\kern-\nulldelimiterspace 
  #1 
  \vphantom{\big|} 
  \right|_{#2} 
  }}

\newtheorem{theorem}{Theorem}[section]
\newtheorem{proposition}[theorem]{Proposition}
\newtheorem{lemma}[theorem]{Lemma}
\newtheorem{corollary}[theorem]{Corollary}
\newtheorem{definition}[theorem]{Definition}
\newtheorem{remark}[theorem]{Remark}
\newtheorem{example}[theorem]{Example}

\graphicspath{{figures/}}
\title{Zeta functions of abstract isogeny graphs and modular curves}

\author{Jun Bo Lau}
\address{KU Leuven, Leuven, Belgium}
\email{junbo.lau@kuleuven.be}
\urladdr{junbolau.github.io}

\author{Travis Morrison}
\address{Virginia Tech, Blacksburg, Virginia, USA}
\email{tmo@vt.edu}
\urladdr{travismo.github.io}

\author{Eli Orvis}
\address{University of Colorado Boulder, Boulder, Colorado, USA}
\email{eli.orvis@colorado.edu}
\urladdr{https://euclid.colorado.edu/~wior7645/HTML/home.html}

\author{Gabrielle Scullard}
\address{University of Georgia, Athens, Georgia}
\email{gabrielle.scullard@uga.edu}
\urladdr{https://sites.google.com/view/gabriellescullard/home}

\author{Lukas Zobernig}
\address{Google}
\email{lukas.zobernig@gmail.com}
\urladdr{}

\date{\today}

\thanks{}



\begin{document}
\begin{abstract}

    We introduce a ``non-orientable'' variation of Serre's definition of a graph, which we call an abstract isogeny graph. These objects capture the combinatorics of the graphs $G(p,\ell,H)$, the $\ell$-isogeny graphs of supersingular elliptic curves with $H$-level structure. In particular they allow for the study of non-backtracking walks, primes, and  zeta functions. We prove an analogue of Ihara's determinant formula for the zeta function of an abstract isogeny graph. 
    For $B_1(N) \subseteq H \subseteq B_0(N)$ and $p > 3$, we use this formula to relate the Ihara zeta function of $G(p,\ell,H)$ to the Hasse-Weil zeta functions of the modular curves $X_{H, {\FF_{\ell}}}$ and $X_{H \times B_0(p), \FF_{\ell}}$. 
    As  applications, we give an explicit formula relating point counts on $X_0(pN)_{\FF_{\ell}}$ and $X_0(N)_{\FF_{\ell}}$ to cycle counts in $G(p,\ell,B_0(N))$ and  prove that the number of non-backtracking cycles of length $r$ in $G(p,\ell,B_0(N))$ is asymptotic to $\ell^r$. 
\end{abstract}

\maketitle

\section{Introduction}\label{sec:Introduction}
Isogeny graphs of supersingular elliptic curves are used for both theoretical and practical reasons, from the fast computation of modular  forms by Mestre~\cite{method-of-graphs}, to the rapidly growing field of isogeny-based cryptography initiated by Charles, Goren, and Lauter~\cite{CGL}. One can obtain many interesting and useful examples of isogeny graphs by considering enhanced elliptic curves with level structure. Fix distinct primes $p$ and $\ell$, an integer $N \geq 1$ coprime to $p\ell$, and a subgroup $H$ of $\GL_2(\ZZ/N\ZZ)$.  We denote by $G(p,\ell,H)$ the $\ell$-isogeny graph of supersingular elliptic curves with $H$-level structure. This graph has as its vertex set a complete set of representatives of isomorphism classes of supersingular curves with $H$-level structure: these are pairs $(E,[\phi])$ where $\phi\colon (\ZZ/N\ZZ)^2\to E[N]$ is an isomorphism and $[\phi]$ denotes the equivalence class of $\phi$ under pre-composition with elements of $H$. The edges of $G(p,\ell,H)$ are morphisms of elliptic curves with $H$-level structure, taken up to post-composition with an automorphism. These graphs generalize $G(p,\ell)$, the supersingular $\ell$-isogeny graph in characteristic $p$, which is given by $G(p,\ell) = G(p,\ell,\GL_2(\ZZ/N\ZZ))$. 

Isogeny graphs are useful in cryptography in large part because they are Ramanujan: their nontrivial eigenvalues are bounded by $2\sqrt{\ell}$.
The Ramanujan property guarantees rapid mixing of non-backtracking walks in $G(p,\ell)$, where a walk is considered non-backtracking if consecutive edges compose to cyclic isogenies. In~\cite{secuer}, the authors use the fact that $G(p, \ell,B_0(N))$ is Ramanujan to construct a zero-knowledge proof of knowledge of an isogeny, and a notion of non-backtracking walks in $G(p,\ell,H)$ for general $H$ may have other cryptographic applications. Questions related to mixing properties in isogeny graphs with level structure have attracted interest recently, especially in the paper of Codogni and Lido~\cite{Codogni-Lido}. This paper studied spectral properties of $G(p, \ell, H)$, showing, for example, that $G(p, \ell, H)$ is Ramanujan whenever $H$ has surjective determinant. 
Codogni and Lido discuss non-backtracking walks only when $\ell \in H$, whereas we initiate the study of such walks in full generality.

Non-backtracking walks are also needed to define the Ihara zeta function of a graph $G$. The Ihara zeta function of $G$ is a meromorphic function whose poles encode information about $G$, such as its Euler characteristic and its spectrum. In the classical definition of the zeta function of a graph, non-backtracking walks are defined with respect to a function $J$ on the edge set that is required to be a fixed-point free involution that swaps the source and target of an edge. After making a non-canonical choice, the dual of an isogeny induces a natural choice of the function $J$ for $G(p, \ell, H)$, but this function can fail all three of the conditions required above!

To study non-backtracking walks and zeta functions of isogeny graphs with level structure, we introduce the notion of an {\bf abstract isogeny graph}.
We define an abstract isogeny graph $\Gamma$ as two finite sets $X$ and $Y$, of {\bf vertices} and {\bf edges}, together with {\bf source} and {\bf target} functions $s,t\colon Y\to X$, and functions $J\colon Y\to Y$ and $L\colon X\to X$ satisfying the following property: for all $y \in Y$, $s(Jy)=t(y)$ and $t(Jy)=Ls(y)$.
This definition generalizes the definition of a graph given by Serre[\S 2.1]~\cite{trees} and Bass~\cite[1.1]{BassGroupActionsOnTrees}, and is sufficiently general to capture the combinatorial data of isogeny graphs with level structure. When $J$ is a fixed-point free involution and the source function $s$ is surjective, $L$ is necessarily the identity and the graph $\Gamma$ is  a graph in the sense of Bass and Serre. In this case, we may partition $Y$ into subsets $S\sqcup JS$. Such a partition is called an {\bf orientation} of $\Gamma$~\cite[\S2.1]{trees}. We therefore call an abstract isogeny graph {\bf orientable} if $J$ is a fixed-point free involution and $L$ is the identity, and abstract isogeny graphs can be viewed as graphs that may be ``non-orientable.''  The maps $J$ and $L$ allow one to define a non-backtracking walk in $\Gamma$ as a sequence of edges $y_1,y_2,\ldots,y_n$ such that $y_{k+1}\not=Jy_k$. 

\subsection{Zeta functions of abstract isogeny graphs}\label{subsec:intro-zeta-abstract-isos}
Let $G$ be a finite graph. Following  ~\cite{Ihara-discrete-PSL2,Sunada-Lfunctions, Hashimoto-det-formula, Bass-IharaSelbergzeta}, one associates a zeta function $\zeta_G(u)$ to $G$ by taking the Euler product 
\begin{equation}\label{eqn:IharaZeta-EulerProduct}
\zeta_G(u)\coloneqq \prod_{[C]}(1-u^{\nu(C)})^{-1},    
\end{equation}
over all primes of $G$. A {\bf prime} of $G$ is an equivalence class of cyclic rotations of a non-backtracking, closed, tailless walk that is not a repeat of another, shorter cycle some number of times. This is a discrete analogue of a simple closed geodesic,  so the definition follows the spirit of Selberg's zeta function of a hyperbolic surface.    The degree $\nu(C)$ of a prime $[C]$ is the number of edges in the cycle $C$. 

If $G$ is undirected, connected, and $(\ell+1)$-regular, then Ihara shows in~\cite[Theorem 2]{Ihara-discrete-PSL2}\footnote{See~\cite[Proposition F]{Sunada-Lfunctions} where Sunada credits this result to Ihara, who proved the analogous statement for the zeta function associated to a subgroup of $\PSL_2(\QQ_\ell)$.} that $\zeta_G(u)$ is a rational function, and in particular,
\begin{equation}\label{eqn:IharaDetFormula}
\zeta_G(u) = \frac{(1-u^2)^{\chi(G)}}{\det(1-uA+\ell u^2)},
\end{equation}
where $\chi(G) = \#\vertices(G)-\#\edges(G)$ is the Euler characteristic of $G$ and $A$ is the adjacency matrix of $G$. Equation~\eqref{eqn:IharaDetFormula} is known as the {\bf Ihara determinant formula}. 

After introducing abstract isogeny graphs, we initiate the study of their zeta functions. 
We define primes and the zeta function of an abstract isogeny graph using non-backtracking cycles with respect to the $J$ function, and prove an analogue of Equation \eqref{eqn:IharaDetFormula} for abstract isogeny graphs in Theorem~\ref{thm:IharaZetaFormula-nonorientable}.  This theorem is in fact a strict generalization of Ihara's determinant formula in that it recovers \eqref{eqn:IharaDetFormula} when the graph is orientable. 

We state two special cases of Theorem~\ref{thm:IharaZetaFormula-nonorientable} that are of interest in the study of isogeny graphs below. The first applies to the case that $\Gamma=(X,Y,s,t,J,L)$ is a finite abstract isogeny graph such that $J$ is an involution with $r$ fixed points and $L$ is the identity map. 
We associate two orientable graphs, denoted $\Gamma^{+1}$ and $\Gamma^{-1}$, to $\Gamma$. The graph $\Gamma^{+1}$ is  obtained by deleting the fixed points of $J$, and the graph $\Gamma^{-1}$ is obtained by duplicating them. The zeta function of the abstract isogeny graph $\Gamma$ encodes the topology of both graphs in the orders of its poles at $1$ and $-1$, respectively: 

\begin{theorem}\label{thm:IharaFormulaInvolutionCase}
    Let $\Gamma=(X,Y,s,t,J,L)$ be an abstract isogeny graph of degree $\ell+1$ and assume $J$ is an involution with $r$ fixed points. Let $A$ be the adjacency operator for $\Gamma$.  Then  $\chi(\Gamma^{-1})=\chi(\Gamma^{+1})-r$ and 
    \[
    \zeta_{\Gamma}(u) = \frac{(1-u)^{\chi(\Gamma^{+1})}(1+u)^{\chi(\Gamma^{-1})}}{\det(1-uA+u^2\ell)}.
    \]

\end{theorem}
This theorem is originally due to Zakharov~\cite{Zakharov}, who gave an equivalent formula for the zeta function of a ``graph with legs:'' a graph in which $J$ is an involution, possibly with fixed points, called legs. Another interesting case of our theorem is the following:

\begin{theorem}\label{thm:IharaFormulaPermutationCase}
    Let  $\Gamma=(X,Y, s, t, J,L)$ be an abstract isogeny graph with the property that $J$ and $L$ are permutations whose cycle decompositions consists entirely of cycles of length $2m$ and $m$, respectively. Then 
\[
\zeta_{\Gamma}(u) = \frac{(1-(-1)^mu^{2m})^{\#X/m}(1-u^m)^{-\#Y/2m}}{\det(1-uA+u^2\ell L)}.
\]
\end{theorem}

Theorem~\ref{thm:IharaFormulaInvolutionCase} applies to the $\ell$-isogeny graphs $G(p,\ell)$ and $G(p,\ell,B_0(N))$ in characteristic $p\equiv 1\pmod{12}$. 
Theorem~\ref{thm:IharaFormulaPermutationCase} applies to the graphs $G(p,\ell,B_1(N))$ for $p\equiv 1\pmod{12}$ and $N>3$ and $(-\ell | p ) =1$. 

\subsection{Relationship to zeta functions of modular curves}
In many cases, we are also able to give a formula relating the zeta function of an abstract isogeny graph to zeta functions of modular curves. 
 Given a curve $C$ over a finite field $k$, let $Z(C,u)$ denote its Hasse-Weil zeta function. Formulas relating the zeta functions of isogeny graphs in characteristic $p\equiv 1\pmod{12}$ with those of modular curves date back to work of Hashimoto~\cite{Hashimoto-det-formula}, and have recently attracted attention in the work of Sugiyama~\cite{Sugiyama-zetas} and Lei-M\"uller~\cite{Lei-Muller}. Let $H$ be a subgroup of $\GL_2(\ZZ/N\ZZ)$ such that $B_1(N)\subseteq H \subseteq B_0(N)$. In Theorem~\ref{thm:IharaZetaModularCurveProduct} we give an explicit product formula, in arbitrary characteristic $p$, relating $Z(X_{H,\FF_\ell}, u)$ to the Ihara zeta function of $G(p, \ell, H)$. 
 
We now give an example of these results in a case where the function $L$ is not the identity. For simplicity, assume $p\equiv 1\pmod{12}$ and $N>2$. The map $L$ on 
the graph of supersingular elliptic curves with $H$-level structure
corresponds to the diamond operator $\langle \ell\rangle$:
\(
L(E,[\phi])=(E,[\ell\phi]).
\)
The map $J$ is induced by the dual map on morphisms. The assumption that $p\equiv 1\pmod{12}$ makes the map $J$ a permutation. Now assume $H=B_1(N)$. If we further assume $(-\ell | p) =1$ then every cycle in the cycle decomposition of $J$ has length equal to $2m$, where $m$ is the smallest positive integer such that $\ell^m\equiv 1$ or $-1$ modulo $N$. The number of vertices of $G(p,\ell,B_1(N))$ is $\#X = 12^{-1}(p-1)N^2\prod_{q|N}(1-q^{-2})$.

\begin{corollary}\label{cor:IntroductionZetaProduct}
    Let $N>3$, $p\equiv1\pmod{12}$, and define $e = \frac{N^2(p-1)}{24m}\prod_{q|N}(1-q^{-2})$. Then
    \[
    Z(X_{B_1(N)\times B_0(p),\FF_\ell},u)Z(X_1(N)_{\FF_\ell},u)^{-2}\zeta_{G(p,\ell,B_1(N))} = (1-(-1)^mu^{2m})^{e}(1-u^m)^{-(\ell+1)e/2}
    \]
\end{corollary}
This is a special case of Theorem \ref{thm:IharaZetaModularCurveProduct}, which relates the zeta function of the isogeny graph of supersingular elliptic curves with $H$-level structure to the zeta functions of $X_{H}$ and $X_{H\times B_0(p)}$ over $\FF_{\ell}$ for any $B_1(N)\subseteq H \subseteq B_0(N)$. 

\subsection{Applications to isogeny graphs and modular curves} After proving Theorems \ref{thm:IharaZetaFormula-nonorientable} and \ref{thm:IharaZetaModularCurveProduct}, we present applications of these results to the study of supersingular isogeny graphs and the study of modular curves over finite fields. One application is an improvement on an asymptotic for the number of non-backtracking cycles in $G(p,\ell)$, which we give in Theorem \ref{thm:CycleAsymptoticsThm}. This asymptotic was first proven by Arpin, Chen, Lauter, Schiedler, Stange, and Tran in~\cite{OrientationsAndCycles}, and recently re-proven using very different methodology by Aycock and Kobin~\cite{AycockKobin}. In Theorem \ref{thm:CycleAsymptoticsThm}, we extend this result to the case of $B_0(N)$-level structure, while also removing the assumption that $p \equiv 1 \pmod{12}$.

Another application of Theorems \ref{thm:IharaZetaFormula-nonorientable} and \ref{thm:IharaZetaModularCurveProduct} is a formula relating point counts on $X_{H, \FF_\ell}$ to cycle counts in supersingular isogeny graphs, given in Corollary \ref{cor:ModularCurvePointCountCor}. These cycle counts have been studied before, and there exist explicit formulas counting these cycles in terms of imaginary quadratic class numbers \cite{OrientationsAndCycles}. Using these formulas, we obtain the following theorem.
\begin{theorem}\label{thm:intro-modcurve-point-count}
Let $p, \ell$ be distinct primes and $r > 2$ such that $\ell^r < p$. Let $G \coloneqq G(p,\ell)$, and $G^{\pm 1}$ be the orientable graphs associated to $G$ in Definition~\ref{def:OrientableGraphs}. Then we have that
\[\#X_0(p)(\FF_{\ell^r}) = 2(1 + \ell^r) - \chi(G^{+1}) + (-1)^{r - 1}\chi(G^{-1}) - 2\sum_{n \mid r} \sum_{\mathcal{O} \in \mathcal{I}_n} h(\mathcal{O}),\]
where $\mathcal{I}_n$ is an explicit set of imaginary quadratic orders.
\end{theorem}
In Appendix~\ref{app:ExplicitEulerChars}, we make the formula in Theorem~\ref{thm:intro-modcurve-point-count} more explicit by giving formulas for $\chi(G(p,\ell, B_0(N))^{\pm 1})$ in terms of imaginary quadratic class numbers and elementary conditions. 

We note that these applications are only accessible because the zeta functions we define count non-backtracking cycles in the sense studied in previous papers on supersingular isogeny graphs. They could not be obtained therefore, from the formulas relating zeta functions of isogeny graphs and modular curves given by Lei and M\"uller or Sugiyama, as we discuss in Sections \ref{sec:Background} and \ref{sec:Applications}. 

\subsection{Organization} In Section~\ref{sec:Background} we give a brief overview of the background required on Ihara zeta functions, modular forms, and the recent papers of Sugiyama and Lei-M\"uller on Ihara zeta functions of superingular isogeny graphs \cite{Sugiyama-zetas, Lei-Muller}. In Section~\ref{sec:non-orientable-graphs} we define abstract isogeny graphs and the orientable graphs associated to abstract isogeny graphs. We prove our version of the Ihara determinant formula for abstract isogeny graphs in Section~\ref{sec:IharaDeterminantFormula}. Section~\ref{sec:abstract-iso-graphs-examples} begins the connection to modular curves by constructing the graphs $G(p,\ell, H)$ as abstract isogeny graphs. In Section~\ref{sec:RelationToModularCurves} we prove a product formula relating the Ihara zeta function of $G(p,\ell, H)$, for $B_1(N) \leq H \leq B_0(N)$ as an abstract isogeny graph to Hasse-Weil zeta functions of modular curves. Finally, in Section~\ref{sec:Applications} we discuss applications to point counting for modular curves and asymptotics for cycles in $G(p,\ell,H)$. We give formulas for the Euler characteristics of the orientable graphs associated to $G(p,\ell,B_0(N))$ in Appendix~\ref{app:ExplicitEulerChars}, and we study the topology of the CW complex associated to an abstract isogeny graph in Appendix~\ref{app:realization}.

\subsection{Acknowledgments}

This project was initiated at the Banff International Research Statement (BIRS) workshop 23w5132 Isogeny Graphs in Cryptography. The authors would like to thank BIRS for their support. The second author was supported in part by the National Science Foundation under Grant No. CNS-2340564. The second author acknowledges support of the Institut Henri Poincar\'e (UAR 839 CNRS-Sorbonne Universit\'e) and LabEx CARMIN (ANR-10-LABX-59-01). The third author was supported by NSF-CAREER CNS-1652238 (PI Katherine E. Stange), as well as an IBM PhD Fellowship while working on this paper. The fourth author was supported by  National Science Foundation Grant CNS-2001470 (PI Kirsten Eisentr\"ager).

\section{Background}\label{sec:Background}
In this section, we discuss prior work on zeta functions of graphs, and in particular, the relationship between the zeta function of an isogeny graph and the zeta function of a modular curve.

\subsection{Ihara zeta functions}
The definition of a graph that most naturally lends itself to the definition and study of zeta functions is due to Serre~\cite{trees}: A {\bf graph} $\Gamma$ consists of the data of a set of {\bf vertices} $X$, {\bf edges} $Y$, a function $(s,t)\colon Y\to X\times X$ giving the {\bf source} and {\bf target} of an edge, and a fixed-point free involution $J\colon Y\to Y$, the {\bf orientation reversal map}, which satisfies $s(Jy)=t(y)$. 

A {\bf walk} in $\Gamma$ is a sequence of edges $(y_1,\ldots,y_n)$ where $s(y_{i+1})=t(y_i)$; the walk is {\bf reduced} or {\bf non-backtracking} if $Jy_i\not= y_{i+1}$. It is {\bf closed} if $s(y_1)=t(y_n)$ and it {\bf has a tail} if $Jy_n = y_1$. A {\bf primitive cycle}  is a non-backtracking, closed, tailless walk that is not a repeat of a shorter cycle. A {\bf prime} of $\Gamma$ is the equivalence class of a primitive cycle under cyclic permutations:
\[
[P] = \{y_1y_2\cdots y_n,\; y_2\cdots y_ny_1, \ldots,\;y_ny_1 \cdots y_{n-1}\}.
\]
The {\bf zeta function} of $\Gamma$ is \[
\zeta_G(u) \coloneqq \prod_{[P]} (1-u^{\nu(P)})^{-1},
\]
where the product is over all primes, and $\nu([P])$ is the length of the cycle $P$. 

Various properties of a connected $(\ell+1)$-regular graph $G$ can be understood from $\zeta_G$. 
The order of the pole of $\zeta_G$ at $1$ is $1-\chi(G)$, the rank of the fundamental group of the CW complex associated to $G$. Sunada~\cite{Sunada-Lfunctions} proved $G$ is {\bf Ramanujan}, meaning the nontrivial eigenvalues of the adjacency operator of $G$ have absolute value at most $2\ell^{1/2}$,  if and only if all nontrivial poles of $\zeta_G$ have magnitude $\ell^{-1/2}$.

\subsection{Prior work}
One way to prove that a given graph $G$ is Ramanujan is to prove that the nontrivial poles of $\zeta_G$ are zeroes of the Hasse-Weil zeta function of a curve over $\FF_{\ell}$. The Riemann hypothesis, proved by Weil for curves~\cite{Weil-number-solutions}, then implies that the nontrivial poles of $\zeta_G(u)$ have modulus $\ell^{-1/2}$. 
This is a combinatorial analogue to the approach, completed by Deligne, to bounding eigenvalues of the Hecke operator $\widetilde{T}_{\ell}$ acting on a space of weight $k$ cusp forms, from which the Ramanujan-Petersson conjecture follows.

Let $\widetilde{T}_{\ell}$ denote the $\ell$th Hecke operator acting on weight $2$ cusp forms for $\Gamma_0(p)$. Eichler proved in~\cite{Eichler-Quaternare} that for almost all $\ell$, the {\em Hecke polynomial} $\det(1-\widetilde{T}_\ell u + \ell u^2)$ is divisible by  the $L$-polynomial of the Hasse-Weil zeta function of $X_0(p)_{\FF_{\ell}}$ . Igusa proves in~\cite{Igusa-kroneckerian} 
 that Eichler's result holds  
for all $\ell\not=p$. Combining this with Weil's proof of the Riemann hypothesis for curves over a finite field in~\cite{Weil-number-solutions}, one deduces the Ramanujan-Petersson conjecture for $\widetilde{T}_{\ell}$ acting on $S_2(\Gamma_0(p))$. Since $\widetilde{T}_{\ell}$ can be identified with the adjacency operator of $G(p,\ell)$, one also can deduce that $G(p,\ell)$ is Ramanujan. This connection is made explicit by Hashimoto~\cite{Hashimoto-det-formula} who relates the zeta function of  certain quotients of $\PSL_2(\QQ_{\ell})$ to the above Hecke polynomial and zeta function.

The recent work of Sugiyama in~\cite{Sugiyama-zetas} and Lei and M\"uller~\cite{Lei-Muller} gives explicit formulas relating the zeta functions of certain supersingular $\ell$-isogeny graphs and of modular curves $X_0(N)$ over $\FF_{\ell}$. 
As mentioned previously, supersingular $\ell$-isogeny graphs are not necessarily graphs as defined above. This complicates the study of counting non-backtracking walks and primitive cycles. Following Kohel~\cite[Chapter 7]{kohel}, for example, one defines the isogeny graph $G(p,\ell)$ as the graph whose vertex set is a complete set of representatives $X=\{E_1,\ldots,E_n\}$ of the isomorphism classes of supersingular elliptic curves over $\overline{\FF}_p$, and whose directed edges from $E_i$ to $E_j$ are a choice of representative of the set $\Aut(E_j)\phi$ for each degree $\ell$ isogeny $\phi\colon E_i\to E_j$. Let $Y$ denote the edge set of $G(p,\ell)$. The reversal map $J\colon Y\to Y$ is defined by mapping an edge $\phi\colon E_i\to E_j$ to the edge given by the chosen isogeny in $\Aut(E_i)\widehat{\phi}$. The map $J$ can fail to be both injective and surjective. When $p\equiv 1\pmod{12}$, the automorphism group of each supersingular curve $E_i$ is $\{[\pm1]\}$, so the map $J$  is an involution. However, $J$ can have fixed points, so $G(p,\ell)$ may not be a graph in the above sense even when $p\equiv 1\pmod{12}$. 

In~\cite{Sugiyama-zetas}, Sugiyama defines the zeta function of $G(p,\ell)$ via Equation~\ref{eqn:IharaZeta-EulerProduct}, and in~\cite{Lei-Muller}, Lei and M\"uller define the zeta function of the $\ell$-isogeny graph of supersingular elliptic curves with Borel level structure via the Ihara determinant formula, Equation~\ref{eqn:IharaDetFormula}.  
When $p\equiv 1\pmod{12}$, the ``Euler characteristic'' $\chi(G(p,\ell))$ of $G(p,\ell)$ would be 
\[
\frac{(p-1)}{12}-\frac{(\ell+1)(p-1)}{24} = \frac{(p-1)(1-\ell)}{24}
\]
which may not be an integer if $J$ has fixed points. The Ihara determinant formula would give \begin{equation}\label{eq:Sugiyama-ihara-zeta-claim}
\zeta_{G(p,\ell)}(u) = \frac{(1 - u^2)^{(p-1)(1 - \ell)/24}}{\det(1 -uA+ \ell u^2)},
\end{equation}
but if $(p-1)(1-\ell)/24$ is not an integer, then this is not a rational function and moreover is not the generating function for the number of non-backtracking tailless cycles in the graph. The zeta function we define in Section~\ref{sec:IharaDeterminantFormula} and associate to $G(p,\ell)$ in Section~\ref{sec:abstract-iso-graphs-examples}, however, is a rational function and is the desired generating function for the number of non-backtracking tailless cycles in the graph. 

\begin{example}\label{example:SugiyamaMistake} 
Let $p = 13$ and $\ell = 2$. There is only one supersingular $j$-invariant, so all three isogenies of degree two are endomorphisms of this curve. The Brandt matrix $B(2)$ is thus the $1 \times 1$ matrix $[3]$. The exponent $(13-1)(1-2)/24 = -1/2$ is not an integer, so Equation \eqref{eq:Sugiyama-ihara-zeta-claim} does not give a rational function for $\zeta_{G(p,\ell)}(u)$.
In the isogeny graph $G(13, 2)$, there is only one self-dual edge. Thus the number of length one non-backtracking tailless cycles is two, the number of length two non-backtracking tailless cycles is six, and the number of length three is eight. Using $f(u)=(1-u^2)^{-1/2}/\det(1-uA+2u^2)$ one gets $u(\log f(u))' = 3u+6u^2+9u^3+\ldots$ while our definition yields $u(\log \zeta_{G(p,\ell)}(u))' = 2u+6u^2+8u^3+\ldots$, giving the expected counts of non-backtracking tailless cycles. 

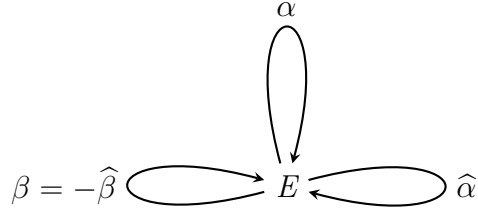
\begin{figure}[h]
    \centering
    \begin{tikzpicture}[>=stealth, thick]

  \node[] (v) {E};

  \path[->]
    (v) edge[loop above, min distance=25mm] node[above] {$\alpha$} (v)
    (v) edge[loop right, min distance=25mm] node[right] {$\widehat{\alpha}$} (v)
    (v) edge[loop left,  min distance=25mm] node[left]  {$\beta=-\widehat{\beta}$} (v);
    (v) node[below] {E};
\end{tikzpicture}
    \caption{The $2$-isogeny graph of supersingular elliptic curves in characteristic $13$. There is a single isomorphism class of supersingular elliptic curves over $\overline{\FF_{13}}$, represented by the curve $E: y^2=x^3+x+4$. The curve $E$ has endomorphisms $\alpha$ and $\widehat{\alpha}$ of degree $3$ with distinct kernels. The corresponding edges are transposed by the map $J$. The endomorphisms $\beta$ and $\widehat{\beta}$ have the same kernel, so $\beta$ corresponds to a fixed point of $J$.}
    \label{fig:placeholder}
\end{figure}
\end{example}

Zakharov studies the zeta function of a graph $G$ in which $J$ is an involution that is allowed to have fixed points and proves an analogue of the Ihara determinant formula for such a graph in Theorem 3.7~\cite{Zakharov}. Zakharov's result lends itself to the calculation of $\zeta_{G(p,\ell)}$ when $p\equiv 1\pmod{12}$ but not for arbitrary $p$.  Corollary~\ref{cor:Ihara-determinant-J-involution} specializes to ~\cite[Theorem 3.7]{Zakharov} when the abstract isogeny graph is a ``graph with legs'' and can be used to calculate $\zeta_{G(p,\ell)}$ for any $p>3$ and $\ell\not=p$. 

In \cite{Lei-Muller}, Lei and M\"uller extended the results Sugiyama in~\cite{Sugiyama-zetas} to the case of isogeny graphs with level structure under the hypothesis that the characteristic $p$ satisfies $p\equiv 1\pmod{12}$. Let $Z(C, u)\coloneqq \prod_P(1-u^{\deg P})^{-1}$ denote the Hasse-Weil zeta function of an algebraic curve $C$ over a finite field $k$, where the product is over the closed points of $C$ and $\deg P=[k(P):k]$. Let  $\zeta_{G(p,\ell,B_0(N))}$ be the Ihara zeta function of the supersingular isogeny graph with Borel $N$-level structure and let $\chi(G(p,\ell,B_0(N)))$ be the Euler characteristic of the graph $G(p,\ell,B_0(N))$, given by the number of vertices minus half the number of edges of $G(p,\ell,B_0(N))$.
\begin{theorem}\cite[Theorem A]{Lei-Muller}\label{thm:LeiMuller-TheoremA}
Let $p\equiv1\pmod{12}$ be a prime, let $\ell\not=p$ be prime, and let $N>0$ be an integer coprime to $p\ell$. 
The following equality holds:
\[
Z(X_0(pN)_{\FF_\ell}, u)Z(X_0(N)_{\FF_\ell},u)^{-2}\zeta_{G(p,\ell,B_0(N))}(u) = (1 - u^2)^{\chi(G(p,\ell,B_0(N)))}
.
\]
\end{theorem}
 Since Equation~\ref{eqn:IharaDetFormula} definition may not give the generating function for the number of primes in supersingular isogeny graphs, Lei-M\"uller's formula does not directly relate the count of primes in supersingular isogeny graphs to point counts on modular curves. Aycock and Kobin make this observation in Remark 5.9 in~\cite{AycockKobin}, where they use different methods from ours to reprove Theorem 7.1 of~\cite{OrientationsAndCycles}; our methods allow us to prove  Theorem ~\ref{thm:CycleAsymptoticsThm}, the generalization of Theorem 7.1 of~\cite{OrientationsAndCycles}  to arbitrary $p>3$ and to isogeny graphs with $B_0(N)$ level structure. In the following section, we introduce abstract isogeny graphs, relaxing the requirement that $J$ be a fixed-point free involution. This lets us define the Ihara zeta function of an isogeny graph with level structure in Section~\ref{sec:abstract-iso-graphs-examples} and in many cases give a formula for it in terms of the Hasse-Weil zeta functions of certain modular curves. In the case of $G(p,\ell,B_1(N))$, the reversal map does not even necessarily satisfy the property that $t(Jy)=s(y)$, showing the need for a new framework to study the combinatorial structure of isogeny graphs.

\section{Abstract isogeny graphs}\label{sec:non-orientable-graphs}
In this section, we extend the definition of a graph used by  Serre~\cite{trees} and Bass~\cite{Bass-IharaSelbergzeta} to accommodate various types of isogeny graphs. 

\subsection{Abstract isogeny graphs}\label{subsec:abstract-iso-graphs-defs}In order to study a combinatorial structure that accurately captures the properties of isogeny graphs, we introduce the following definition: 

\begin{definition}\label{def:abstract-isogeny-graph}
An {\bf (abstract) isogeny graph}  $\Gamma$ consists of two sets, $X = \vertices\Gamma$ and $Y=\edges\Gamma$, a function $Y\to X\times X$ sending an edge $y$ to $(s(y),t(y))$ (the {\bf source} and {\bf target} of the directed edge $y$), a function $J\colon Y\to Y$, the {\bf dual} map, and a function $L\colon X\to X$, such that
\begin{enumerate}
    \item $s(Jy) = t(y)$ for all $y\in Y$, and
    \item $t(Jy) = Ls(y)$ for all $y\in Y$.  
\end{enumerate}

\end{definition}

\begin{remark}\label{rmk:LfunctionMotivation}
The addition of a function $L$ and its properties are motivated by the diamond operator $\langle \ell \rangle$ on vertices of $G(p, \ell, H)$, which sends $(E, [\phi])$ to $(E, [\ell \phi])$.
\end{remark}

\begin{definition}
     Suppose $\Gamma$ is an abstract isogeny graph. For a vertex $x \in \vertices \Gamma$, we call $d(x) \coloneqq \#\{y \in Y : s(y) = x\}$ the {\bf degree} of $x$. If $d(x)$ is finite and independent of $x\in \vertices\Gamma$, we call $\Gamma$ {\bf regular}; if this constant is $d$, we call $\Gamma$ {\bf $d$-regular}. If $J$ is a fixed-point free involution and $L$ is the identity, we will say $\Gamma$ is an {\bf orientable isogeny graph}. Otherwise, we say that $\Gamma$ is {\bf non-orientable}. An {\bf orientation} on $\Gamma$ is a subset $Y_+\subset Y$ such that $Y$ is the disjoint union of $Y_+$ and $J(Y_+)$. 

\end{definition}

If $\Gamma$ is an orientable abstract isogeny graph, an orientation always exists. When $\Gamma$ is orientable, we have 
\[
t(Jy)=s(J^2y)=s(y)
\]
for all $y\in Y$, so $L$ is the identity. An orientable isogeny graph is simply a graph according to Serre and Bass. An abstract isogeny graph such that $L$ is the identity and $J$ is an involution, possibly with fixed points, is a {\bf graph with legs} in~\cite[Section 2.1]{Zakharov}.

\subsection{The orientable graphs associated to an abstract isogeny graph}\label{subsec:orientable-graphs-associated}

Given an abstract isogeny graph $\Gamma$ with vertex set $X$ and edge set $Y$, we define two orientable graphs $\Gamma^{+1}$ and $\Gamma^{-1}$. The choice of notation is explained in the next section, where we interpret the order of the poles at $\pm 1$ of $\zeta_G(u)$ as the Euler charactersitcs of $\Gamma^{\pm 1}$, see Remark~\ref{rmk:interpretation-of-Euler-chars}. 

Let $\sim_X$ be the smallest equivalence relation on $X$ so that $x\sim Lx$ for all $x\in X$.  Let $\sim_Y$ be the smallest equivalence relation on $Y$ such that $y\sim_Y J^2y$ for all $y\in Y$. 
Let $X'= X/{\sim_X}$ and let $Y'=Y/{\sim_Y}$ denote the equivalence classes of $\sim_X$ and $\sim_Y$. Abusing notation, we define the source and target functions $s,t : Y' \to X'$ by $s([y]) = [s(y)]$ and $t([y]) = [t(y)]$; these maps are well-defined. 

The map $J\colon Y\to Y$ induces a well-defined map
\begin{align*}
    Y'&\to Y' \\ 
    [y] &\mapsto [Jy]
\end{align*} that we will also denote by $J$. In particular, we have 
\[
s(J[y]_{{\sim_Y}})=[s(Jy)]_{{\sim_X}} = [t(y)]_{{\sim_X}}=t([y]_{\sim_Y}) \]
and 
\begin{equation}\label{eq:Induced-abstract-iso-eqn}
t(J[y]_{{\sim_Y}})=t([Jy]_{{\sim_Y}})=[t(Jy)]_{{\sim_X}}=[L(s(y))]_{\sim_X}=[s(y)]_{{\sim_X}}=s([y]_{{\sim_Y}}).
\end{equation}
From the definition of ${\sim_Y}$, the map $J$ is an involution on $Y/{\sim_Y}$: 
\[
J^2[y] = [J^2y] = [y].
\]
The vertex set of $\Gamma^{\pm1}$ is $X/{\sim_X}$. The data $(X/{\sim_X},Y/{\sim_Y}, s,t,J)$ defines a graph with legs. We now construct the orientable graphs associated to $\Gamma$. 
We first define edge sets $Y^{+1}\coloneqq Y' - \{[y]: J[y]=[y]\}$ and  $Y^{-1}\coloneqq Y'\sqcup \{[y]:J[y]=[y]\}$. Let $J_{+1}$ denote the restriction of $J$ to $Y^{+1}$. To define an involution on $Y^{-1}$,  we define $J_{-1}$ so that $J_{-1}[y]=J[y]$ if $J[y]\not=[y]$ and $J_{-1}$ swaps $[y]\in \{J[y]=[y]\}$ with its copy in $Y^{-1}$. We set $\Gamma^{+1}$ equal to the graph $(X', Y^{+1}, s, t, J)$, and define $\Gamma^{-1}$ similarly. Note that the function $L$ induces a function on the vertices of $\Gamma^{\pm 1}$, which we also denote by $L$, and which is always the identity. Equation \eqref{eq:Induced-abstract-iso-eqn} shows that the functions $J$ and $L$ on $\Gamma^{\pm 1}$ satisfy the properties of an abstract isogeny graph.

\begin{definition}\label{def:OrientableGraphs}
Let $\Gamma = (X, Y, s, t, J, L)$ be an abstract isogeny graph. We define $\Gamma^{+1}$ to be the abstract isogeny graphs with vertex set $X / \sim_X$, edge set $Y^{+1}$, and functions $s, t, J, L$ induced from the corresponding functions on $\Gamma$ as described above. We define $\Gamma^{-1}$ similarly, and call $\Gamma^{+1}$ and $\Gamma^{-1}$ the {\bf orientable graphs associated to $\Gamma$}.
\end{definition}
 Note that if $\Gamma$ is orientable then $\Gamma^{\pm1}=\Gamma$.

\section{The Ihara determinant formula for abstract isogeny graphs}\label{sec:IharaDeterminantFormula}
We will now define the Ihara zeta function of an abstract isogeny graph. We then prove an analogue of Ihara's determinant formula in Theorem~\ref{thm:IharaZetaFormula-nonorientable}.

\subsection{Ihara zeta functions of abstract isogeny graphs}

To study the Ihara zeta function of an abstract isogeny graph, we need to extend the definitions of paths, cycles, and primes to abstract isogeny graphs. This is accomplished in Definitions \ref{def:paths-cycles-abstract-iso-graphs} and \ref{def:abstract-iso-graph-defs}.
 
\begin{definition}\label{def:paths-cycles-abstract-iso-graphs}
Let $\Gamma = (X, Y, s, t, J, L)$ be an abstract isogeny graph. A {\bf path} in a graph $\Gamma$ is 
a sequence $(y_1,\ldots,y_n)$ of edges $y_i\in \edges\Gamma=Y$ where the edges $y_i$  satisfy $s(y_{i+1})=t(y_i)$ for $1\leq i \leq n-1$. 
We will often denote the path by $y_1\cdots y_n$. 
A {\bf cycle} in $\Gamma$ is a path $(y_1,\ldots,y_n)$ such that $t(y_n)=s(y_1)$.   A graph is {\bf connected} if every ordered pair of vertices are connected by a path. 
\end{definition}

\begin{definition}\label{def:abstract-iso-graph-defs}
    Let $\Gamma$ be an abstract isogeny graph.  If $P=y_1\cdots y_n$ is a path 
    in $\Gamma$, define the {\bf length} of $P$ 
    to be $\nu(P)\coloneqq n$. The path $P$ is {\bf non-backtracking} 
    if $Jy_i\not=y_{i+1}$ for $1\leq i < n$. 
    A cycle {\bf has a tail} if $y_1=Jy_s$. A cycle $C$ in $\Gamma$ is {\bf primitive}  if it is non-backtracking, has no tail, and  $C\not=C'^f$ for any $f>1$ and any cycle $C'$. A {\bf prime} of $G$ is an equivalence class of a primitive cycle under cyclic permutation:
    \[
    [C]=\{y_1\cdots y_s,\quad y_2\cdots y_sy_1, \quad\ldots, \quad y_sy_1\cdots y_{s-1}\}. 
    \]
\end{definition}

\begin{remark}\label{rmk:sd-loops-exist}
    Let $\Gamma$ be an abstract isogeny graph and suppose there is $y\in \edges\Gamma$ satisfy $Jy=y$. Then $y$ is a path beginning and ending at $s(y)=t(y)$ and is therefore a cycle. Since the cycle $y$ has a tail, however, it does not yield a prime of length $1$. Our definition of a primitive cycle in an abstract isogeny graph matches the corresponding definition for a graph with legs in~\cite{Zakharov}.
\end{remark}
We now define the Ihara zeta function for an abstract isogeny graph.

    \begin{definition}\label{def:ihara-zeta-for-abstract-iso-graphs}
            The {\bf Ihara zeta function} of an abstract isogeny graph $\Gamma$ is given by the Euler product
    \[
    \zeta_{\Gamma}(u)\coloneqq \prod_{[P]}(1-u^{\nu(P)})^{-1}
    \]
    where the product is taken over primes in $\Gamma$. 
    \end{definition}
\subsection{Operators on vector spaces defined by abstract isogeny graphs}

In this section, we define operators on certain vector spaces associated to an abstract isogeny graph. These operators will be used to prove Theorem \ref{thm:IharaZetaFormula-nonorientable}. Our presentation closely follows Bass' proof of the Ihara determinant formula~\cite{Bass-IharaSelbergzeta}. 
We remove the assumption that $J$ is a fixed-point free involution, replace the assumption $t(Jy)=s(y)$ with $t(Jy)=Ls(y)$, and make the necessary modifications to Bass' proof to apply the results to our abstract isogeny graphs.

Let $\Gamma=(X,Y,s,t,J,L)$ be an abstract isogeny graph. We assume $\Gamma$ is {finite}. Let $\CC^{X}$ and $\CC^{Y}$ be the finite-dimensional vector spaces with bases given by $X=\vertices\Gamma$ and $Y=\edges\Gamma$, respectively. Define the following operators  via their action on the bases $X$ and $Y$ of $\CC^{X}$ and $\CC^{Y}$, respectively:
\begin{align*}
    J\colon \CC^{Y}& \to \CC^{Y} & S\colon \CC^{Y}&\to \CC^{X}  & T\colon \CC^{Y}&\to \CC^{X}  \\ 
    y&\mapsto J(y) & y  &\mapsto s(y) &y &\mapsto t(y) \\
    & & & & & \\
    W_1\colon \CC^{Y}&\to \CC^{Y} & S^*\colon \CC^{X}&\to \CC^{Y}  & L\colon \CC^X &\to \CC^X \\
    y &\mapsto \sum_{\substack{s(y')=t(y) \\ y'\not=J(y)}}y'
    & x &\mapsto \sum_{s(y)=x}y  & x &\mapsto Lx.
\end{align*}
 Also define 
\[
a_{xx'}\coloneqq \# \{e\in s^{-1}(x): t(e)=x'\}
\]
and let $A$ be the adjacency operator of $\Gamma$, i.e. the operator 
\begin{align*}
    A\colon \CC^{X}&\to \CC^{X} \\ 
    x &\mapsto \sum_{y:\, s(y)=x} t(y) = \sum_{x'} a_{xx'} x'.
\end{align*}
Finally, define $D\colon \CC^X\to \CC^X$ via $D(x)=d(x)x$.

\begin{proposition}\label{prop:operator-identities}
    Let $\Gamma=(X,Y,s,t,J,L)$ be an abstract isogeny graph. With the above notation, we have
    \begin{multicols}{2}
    \begin{enumerate}[(a)]
        \item $SJ = T$ \label{prop:op-ids-SJ}
        \item $TJ = LS$ \label{prop:op-ids-TJ}
        \item $TS^* = A$ \label{prop:op-ids-TS*}
        \item $SS^* = D$ \label{prop:op-ids-SS*}
        \item $W_1 + J = S^*T$ \label{prop:op-ids-S*T}
    \end{enumerate}
    \end{multicols}
\end{proposition}

\begin{proof}
    For any edge $y\in Y$ we have $SJy=s(Jy)=t(y)=Ty$, so $SJ$ and $T$ agree on the basis $Y$ for $\CC^Y$ and therefore are equal. Similarly, we have $TJy = t(Jy) = Ly$ for all $y\in Y$.  

    For a vertex $x$ we have 
    \[
    TS^*x = T\left(\sum_{y \colon s(y)=x}y\right) = \sum_{x'} a_{xx'}x' 
    \]
    so $TS^*$ and $A$ agree on the basis $X$. Therefore they are equal as operators. We also have 
    \[
    SS^* x = S\left(\sum_{y\colon s(y)=x}y\right) = d(x)x = Dx.
    \]
Now let $y\in Y$ and calculate
    \[
    (W_1+J)y = \left(\sum_{\substack{s(y')=t(y) \\ y'\not=Jy}}y'\right) + J y = \sum_{y' \colon s(y')=t(y)}y' = S^*t(y) = S^*Ty.
    \]
    
\end{proof}

\begin{definition}\label{def:edge-zeta-for-abstract-iso-graphs}
For $y,y'\in Y$, let $w_{yy'}\in \CC$ satisfy $w_{yy'}=0$ if $y'=Jy$ or $s(y')\not=t(y)$. 
    Define the operator $W\colon \CC^{Y}\to \CC^{Y}$ by defining 
    \[
    Wy = \sum_{y'}w_{yy'}y'. 
    \]
 Define the {\bf norm} of a cycle $C=y_1\cdots y_s$ of length at least $2$  to be 
    \[
    N(C)=N_W(C)\coloneqq w_{y_1y_2}\cdots w_{y_{s-1}y_s}w_{y_sy_1}. 
    \]
    When $C=\{y\}$ is a loop, define $N(C)\coloneqq w_{yy}$. 
    Then the {\bf edge zeta function for $\Gamma$} (with respect to $W$) is 
    given by the Euler product
    \[
    \zeta_{E}(\Gamma,W)\coloneqq \prod_{[P]}(1-N_W(P))^{-1}
    \]
    where again the product is taken over primes of $\Gamma$. 
\end{definition}

\begin{proposition}\label{prop:edgezeta-vs-iharazeta}
    If $W=W_1u$ with $u\in \CC$ (i.e. each nonzero entry of $W$ is defined to be $u$) then 
    \[
    \zeta_{E}(\Gamma,W_1u)=\zeta_{\Gamma}(u). 
    \]
\end{proposition}
\begin{proof}
    For a prime $[P]$ with $P=y_1\cdots y_s$ and $W=W_1u$ we have 
    \[N(P) = w_{y_1y_2}\cdots w_{y_{s-1}y_s}w_{y_sy_1} = u\cdots u = u^{\nu(P)},\] and the result then follows by comparing Definitions \ref{def:edge-zeta-for-abstract-iso-graphs} and \ref{def:ihara-zeta-for-abstract-iso-graphs}.
\end{proof}

\begin{proposition}\label{prop:edge-zeta-formulaW1} The edge zeta function of $\Gamma$ is related to the operator $W_1$, as follows:
    \[
    \zeta_E(\Gamma,W_1u)=\det(\id_{\CC^{Y}}-W_1u)^{-1}.
    \]
\end{proposition}

\begin{proof}
    The argument given by Bass in ~\cite[Theorem 3.3]{Bass-IharaSelbergzeta} works  just as well for abstract isogeny graphs. We include a condensed proof for completeness. Set $W=W_1u$. One shows that 
    \[
    W^ky = \sum_{(yy_1\ldots y_k)} u^ky_k,
    \]
    where the sum is over {non-backtracking} paths beginning with edge $y$ and ending with edge $y_k$, so 
       \[
       \log\zeta_E(\Gamma,W) = \sum_{\substack{C \text{ cycle}\\ \text{no backtracking or tail}}} \frac{u^{\nu(C)}}{\nu(C)} = \sum_{k\geq 1} \frac{\Tr W^k}{k}.
    \]
    Next, compute   
    \[
    \Tr \log(\id_{\CC^{Y}}-W)^{-1} = -\Tr\sum_{k\geq 1}\frac{-W^k}{k} = \Tr\sum_{k\geq 1} \frac{W^k}{k} = \sum_{k\geq 1}\frac{\Tr W^k}{k}.
    \]
    Finally use the fact that $\Tr\log M = \log\det M$ for a matrix $M$ to get 
    \[
    \log \det (\id_{\CC^{Y}}-W)^{-1} = \Tr\log(\id_{\CC^{Y}}-W)^{-1} = \sum_{C}\frac{N(C)}{\nu(C)} = \log \zeta_{E}(\Gamma,W). 
    \]
\end{proof}
We are almost ready to state the Ihara determinant formula for regular abstract isogeny graphs. Before doing so, we need the notation of Definition \ref{def:CycleLengthsAssociated} and the technical lemmas \ref{lem:SelfMapRestrictsToPermutation} and \ref{lem:endo-function-det}.
\begin{lemma}\label{lem:SelfMapRestrictsToPermutation}
    Let $B$ be a finite, nomepty set and let $F\colon B\to B$ be a function. Then there exists a unique nonempty maximal subset $Z$ of $B$ such that $\restr{F}{Z}$ is a permutation. 
\end{lemma}
\begin{proof}
    If $Z$ and $Z'$ are two subsets of $B$ such that $\restr{F}{Z}$ and $\restr{F}{Z'}$ are permutations, then $\restr{F}{Z\cup Z'}$ is too. Thus there exists a unique maximal subset $Z$ such that $\restr{F}{Z}$ is a permutation if there exists any nonempty such subset. We prove such a subset exists by induction on $\#B$. If $B$ is a singleton then $F$ is already a permutation. Suppose the claim holds for $n>1$ and let $\#B=n+1$. If $F$ is surjective, then it is bijective, since $B$ is finite. So suppose $b\in B$ is not in the image of $F$ and apply the inductive hypothesis to $B'=B-\{b\}$ and $F'=\restr{F}{B'}$. 
\end{proof}
\begin{definition}\label{def:CycleLengthsAssociated}
    Let $B$ be a finite, nonempty set and let $F\colon B\to B$ be a function. Let $Z$ be the largest subset of $B$ such that $\restr{F}{Z}$ is a permutation. Define $\restr{F}{Z}$ to be the {\bf permutation associated to $F$} and let $C_k(F)$ denote the number of cycles of length $k$ in the cycle decomposition of $\restr{F}{Z}$. 
\end{definition}

\begin{lemma}\label{lem:endo-function-det}
    Let $B$ be a finite, nonempty set and let $F\colon B\to B$ be a function. Denote by $F$ the operator on $V=\CC^B$ determined by $F$. Let $s\in \CC$. Then 
    \[
    \det(\id_{\CC^B}+sF) = (1+s)^{C_1(F)}\prod_{k>1}(1-s^k)^{C_k(F)}.
    \]
\end{lemma}
\begin{proof}
    Let $I := \id_{\CC^B}$. We must compute $\det(I+sF)$. First, we claim that if $b$ is not in the image of $F$  then the determinant of $I+sF$ is equal to the determinant of $I+sF$ restricted to $\CC^{B-\{b\}}$. Indeed, on $\CC^B=\CC^{\{b\}}\oplus \CC^{B-\{b\}}$, $I+sF$ is block upper triangular with blocks given by the identity on $\CC^{\{b\}}$ and $I+sF$ on $\CC^{B-{\{b\}}}$.
    Thus we may restrict to the largest subset $Z\subseteq B$ such that every $b \in Z$ has a pre-image under $F$ in $Z$. Then $Z$ is also  the largest subset of $B$ such that $F$ acts as a permutation $\sigma_F$ on $Z$: the claim fails for any larger subset because there is an element without a preimage, and it holds on $Z$ because $Z$ is finite. The resulting permutation is unique, and if we order the elements of $Z$ corresponding to the cycle decomposition of the permutation $\sigma_J$, we see we are left with computing the determinant of matrices of the form 
    \[
    \begin{pmatrix}
        1 & 0 & 0 & \ldots & 0 & s\\ 
        s & 1 & 0 & \ldots & 0 & 0 \\
        0 & s & 1 & \ldots & 0 & 0 \\
        0 & 0 & s & \ddots & 0 & 0 \\ 
        \vdots & \vdots & \vdots  & \ddots & 1 & 0 \\ 
        0 & 0 & 0  & \ldots & s & 1
    \end{pmatrix}.
    \]
    This is the $k\times k$ matrix $I+sP_{(123\cdots k)}$ where $P_{\tau}$ denotes the permutation matrix for the permutation $\tau$. Considering the Leibniz formula for the determinant, we see the only nonzero terms correspond to the identity permutation and the permutation $(123\ldots k)$, giving us a factor of $1-s^k$. We also get a factor of $(1+s)$ for each fixed point of $F$, i.e. each cycle of length $1$. The resulting determinant is therefore 
    \[
    \det(I+sF) = (1+s)^{C_1(F)}\prod_{k>1} (1-s^k)^{C_k(F)}.
    \]
\end{proof}

\begin{theorem}[The Ihara determinant formula for regular abstract isogeny graphs]\label{thm:IharaZetaFormula-nonorientable}
Let $\Gamma$ be a finite abstract isogeny graph. Let $Q=D-\id_{\CC^{X}}$ and assume $D$ commutes with $L$. 
    Then 
    \[
    \zeta_{\Gamma}(u) = \frac{(1-u^2)^{C_1(L)}(1+u)^{-C_1(J)}\prod_{k>1}(1-(-1)^ku^{2k})^{C_k(L)}(1-u^k)^{-C_k(J)}}{\det(\id_{\CC^{X}}-Au+ u^2QL)}.
    \]
\end{theorem}

\begin{remark}
    The condition that $D$ commutes with $L$ in the above theorem holds if, for example, the graph is regular. All isogeny graphs we consider in the following sections are regular.
\end{remark}

\begin{proof}
    We adapt the proof of Theorem 3.9 of Bass in~\cite{Bass-IharaSelbergzeta} to the setting of abstract isogeny graphs. Let $n=\#X$. A calculation shows 
    \begin{equation}\label{eqn:matrix-product1}
    \begin{pmatrix}
        \id_{\CC^{X}}-u^2L & 0 \\ uS^* & \id_{\CC^{Y}}-uW_1
    \end{pmatrix}
    \begin{pmatrix}
        \id_{\CC^{X}} & T \\ 0 & \id_{\CC^{Y}}
    \end{pmatrix}
    =\begin{pmatrix}
        \id_{\CC^{X}}-u^2L & T-u^2LT \\ 
        uS^* & uS^*T + \id_{\CC^{Y}}-uW_1 
    \end{pmatrix}
    \end{equation}
    and that  
    \[
    \begin{pmatrix}
        \id_{\CC^{X}} & T-uTJ \\ 0 & \id_{\CC^{Y}}
    \end{pmatrix}
    \begin{pmatrix}
        \id_{\CC^{X}}-uA+u^2QL & 0 \\ uS^* & \id_{\CC^{Y}}+uJ
    \end{pmatrix}
    \]
    is equal to
    \begin{equation}\label{eqn:matrix-product2}
    \begin{pmatrix}
        1-uA+u^2QL+uTS^*-u^2TJS^* & T-u^2TJ^2 \\ 
        uS^* & 1+uJ
    \end{pmatrix}.
    \end{equation}
    We claim that the two matrix products in Equations~\ref{eqn:matrix-product1} and~\ref{eqn:matrix-product2} agree.
    To prove this, we make several uses of Proposition~\ref{prop:operator-identities}.
    The upper-left entries in Equations~\ref{eqn:matrix-product1} and~\ref{eqn:matrix-product2} are equal since 
    \begin{align*}
        \id_{\CC^{X}}-uA + u^2QL +uTS^* -u^2TJS^* &= \id_{\CC^{X}}+u^2QL-u^2TJS^* && \text{Proposition~\ref{prop:operator-identities}(\ref{prop:op-ids-TS*})}\\ 
        &=\id_{\CC^{X}}+u^2(D-\id_{\CC^{X}})L -u^2LSS^* && \text{Proposition~\ref{prop:operator-identities}(\ref{prop:op-ids-TJ})}\\
        &=\id_{\CC^{X}}-u^2L +u^2DL - u^2LD && 
        \text{Proposition~\ref{prop:operator-identities}(\ref{prop:op-ids-SS*})}\\
        &=\id_{\CC^{X}}-u^2L && LD=DL.
    \end{align*}
    The upper-right entries are equal using Proposition~\ref{prop:operator-identities} parts (\ref{prop:op-ids-SJ}) and~(\ref{prop:op-ids-TJ}):
    \[
    T-u^2TJ^2 = T-u^2LSJ = T-u^2LT.
    \]
    We now show the bottom-right entries are equal; this follows from Proposition~\ref{prop:operator-identities}(\ref{prop:op-ids-S*T}):
    \[
    uS^*T+\id_{\CC^Y}-uW_1 = u(J+W_1)+\id_{\CC^Y} -uW_1 = \id_{\CC^Y}+uJ.
    \]
    Thus
    \[
    \det(\id_{\CC^X}-u^2L)\det(\id_{\CC^{Y}}-uW_1) = \det(\id_{\CC^{X}}-uA+ u^2QL)\det(\id_{\CC^Y}+uJ).
    \]
    The claimed formula for $\zeta_{\Gamma}(u)$ now follows by applying  Lemma~\ref{lem:endo-function-det} to $\det(1-u^2L)$ and $\det(1+uJ)$ and using Propositions~\ref{prop:edgezeta-vs-iharazeta} and~\ref{prop:edge-zeta-formulaW1}.
\end{proof}

In many cases, we can give the Ihara zeta function of an abstract isogeny graph $\Gamma$ in a more compact form, using the orientable graphs $\Gamma^{\pm 1}$ defined in Section \ref{sec:non-orientable-graphs}. Two special cases are given in Corollaries \ref{cor:Ihara-determinant-original-recovered} and \ref{cor:Ihara-determinant-J-involution} below.

\begin{corollary}\label{cor:Ihara-determinant-original-recovered}
Let $\Gamma = (X,Y, s, t, J, L)$ be a $d$-regular abstract isogeny graph with $d \geq 1$, and let $Q = D - \id_{\CC^X} = (d-1)\id_{\CC^X}$. Suppose $J$ is a fixed-point free involution. Then we have that
    \[
    \zeta_{\Gamma}(u) = \frac{(1-u^2)^{\chi(\Gamma)}}{\det(\id_{\CC^{X}}-uA+ u^2Q)}.
    \]
\end{corollary}

\begin{proof}
Since $J$ is an involution, we must have that \[s(y) = s(J^2 y) = t(Jy) = Ls(y)\] for all $y \in Y$. Since $\Gamma$ is regular of degree $d \geq 1$, every vertex $x\in X$ is equal to $s(y)$ for some $y\in Y$. It follows that $L$ is the identity on $X$. Thus $C_1(L) = \# X$, and $C_k(L) = 0$ for $k > 1$. Since $J$ is an involution, every cycle has length either one or two, but since $J$ is fixed-point free, $C_1(J) = 0$. Thus $C_2(J) = \frac{1}{2} \# Y$, and so 
\[\zeta_\Gamma(u) = \frac{(1-u^2)^{\# X}(1 - u^2)^{-\frac{1}{2}\#Y}}{\det(\id_{\CC^X} - Au + u^2 QL)} = \frac{(1-u^2)^{\chi(X)}}{\det(\id_{\CC^X} - Au + u^2 Q)}.\]
\end{proof}

When $J$ is an involution with fixed points, Theorem 3.7 of~\cite{Zakharov} gives an analogue of the Ihara determinant formula for $\zeta_{\Gamma}$. We recover Zakharov's theorem and extend it to certain cases where $J$ is not an involution but induces an involution on $Y/{\sim_Y}$. In particular, the following corollary applies to isogeny graphs with Borel level structure. 
\begin{corollary}\label{cor:Ihara-determinant-J-involution}
Let $\Gamma = (X, Y, s, t, J, L)$ be a $d$-regular abstract isogeny graph for $d \geq 1$, and $Q$ be as above. Suppose that the permutation induced by $J$ (see Definition~\ref{def:CycleLengthsAssociated}) is an involution and 
$s(J^2y)=s(y)$ for every edge $y\in Y$. Then we have that 
\[\zeta_\Gamma(u) = \frac{(1 + u)^{\chi(\Gamma^{-1})}(1 - u)^{\chi(\Gamma^{+1})}}{\det(\id_{\CC^X} - uA + u^2 Q)}.\]
\end{corollary}

\begin{proof}
As in the proof of Theorem \ref{cor:Ihara-determinant-original-recovered}, we have that $L$ is the identity.  
Since the permutation induced by $J$ is an involution, it is a product of disjoint transpositions, so $C_k(J) = 0$ for $k \geq 3$. The number of edges in $\Gamma^{+1}$ is $\#\{[y]\in Y/{\sim_Y}: [Jy]\not=[y]\}=2C_2(J)$. The number of edges in $\Gamma^{-1}$ is $\# Y/{\sim_Y} + \#\{[y]: [Jy]=[y]\}=2C_2(J)+2C_1(J)$. Thus $\chi(\Gamma^{+1})=\#X-C_2(J)$ and $\chi(\Gamma^{-1})=\#X-C_2(J)-C_1(J)$.
Theorem \ref{thm:IharaZetaFormula-nonorientable} gives
\begin{align*}
\zeta_\Gamma(u) & = \frac{(1 - u^2)^{\#X}(1 + u)^{-C_1(J)}(1 - u^2)^{-C_2(J)}}{\det(\id_{\CC^X} - uA + u^2 QL)} \\
& = \frac{(1 + u)^{\chi(\Gamma^{-1})}(1 - u)^{\chi(\Gamma^{+1})}}{\det(\id_{\CC^X} - uA + u^2 Q)}. \\
\end{align*}
\end{proof}

\begin{remark}\label{rmk:interpretation-of-Euler-chars}
Corollary \ref{cor:Ihara-determinant-original-recovered} recovers the original Ihara determinant formula, while Corollary \ref{cor:Ihara-determinant-J-involution} interprets the powers of $(1-u)$ and $(1+u)$ occuring in the determinant formula for $\zeta_\Gamma(u)$ as the Euler characteristics of $\Gamma^{-1}$ and $\Gamma^{+1}$, respectively. The order of the pole of $\zeta_\Gamma$ at $1$ is the rank of the fundamental group of the CW complex associated to the abstract isogeny graph $\Gamma$; see Appendix~\ref{app:realization} and Corollary~\ref{cor:realization_homotopy_type}.
\end{remark}

\section{Isogeny graphs as abstract isogeny graphs}\label{sec:abstract-iso-graphs-examples}

In this section, we discuss the construction of $G(p, \ell, H)$ and the $(\ell, \hdots, \ell)$-isogeny graph of superspecial principally polarized abelian varieties as an abstract isogeny graphs. 

\subsection{Isogeny graphs with level structure}\label{sec:level-structure}
We now give a formal introdution to the notion of an isogeny graph with level structure as defined by Codogni--Lido~\cite{Codogni-Lido}. These graphs generalize the supersingular isogeny graphs used in~\cite{method-of-graphs}, including the usual supersingular isogeny graph $G(p,\ell)$. 
\begin{definition}
    Let $N$ be a positive integer and let $H$ be a subgroup of $\GL_2(\ZZ/N\ZZ)$. Let $k$ be a field of characteristic not dividing $N$. Given an elliptic curve $E$ over $k$, a {\bf level-$N$ structure} on $E$ is an isomorphism $\phi\colon (\ZZ/N\ZZ)^2\to E[N]$ and an {\bf $H$-level structure} on $E$ is an equivalence class of isomorphisms $\phi\colon (\ZZ/N\ZZ)^2\to E[N]$, where two isomorphisms are equivalent if they differ by precomposition by an element of $H$. Denote the equivalence class of $\phi$ by $[\phi]_H$ (or just by $[\phi]$, if $H$ is clear from context).  

    Let $(E_1,[\phi_1])$ and $(E_2,[\phi_2])$ be two elliptic curves with $H$-level structures. A {\bf morphism of $H$-level structures} $(E_1,[\phi_1])\to (E_2,[\phi_2])$ is an isogeny $\alpha\colon E_1\to E_2$ such that $[\alpha\circ \phi_1]=[\phi_2]$.
    
\end{definition}
An isomorphism $\alpha \colon E_1\to E_2$ induces a morphism $(E_1,[\phi_1])\to (E_2,[\phi_2])$ if and only if $\phi_2^{-1}\circ \alpha \circ \phi_1\in H$. In this case, the inverse of $\alpha$ is also a morphism:  
    \[
    \phi_2^{-1} \circ \alpha \circ \phi_1 \in H \implies (\phi_2^{-1} \circ \alpha \circ \phi_1)^{-1} = \phi_1^{-1} \circ \alpha^{-1} \circ \phi_2\in H.
    \]
    An automorphism $u$ of $E$ induces an automorphism of $(E,[\phi])$ if it is a morphism $(E,[\phi])\to (E,[\phi])$, i.e. $u\circ\phi = \phi\circ h$ for some $h\in H$. This holds if and only if $\phi^{-1}\circ u \circ \phi \in H$. 

Let  $d$ be an integer coprime to $N$ and let $\phi\colon (\ZZ/N\ZZ)^2 \to E[N]$ be a level-$N$ structure on $E$. Since $d$ is prime to $N$, the function $d\phi\coloneqq [d]\circ \phi\colon (\ZZ/N\ZZ)^2\to E[N]$ is also a level-$N$ structure. Let $X$ denote a complete set of representatives for the isomorphism classes of supersingular elliptic curves with $H$-level structure.  We have the {\bf diamond operator} $\langle d \rangle \colon X\to X$ that maps $x=(E,[\phi])$ to $\langle d \rangle x\coloneqq x'\simeq (E,[d\phi])$ where $x'\in X$.

Let $\alpha\colon(E,[\phi])\to(E',[\phi'])$ be a morphism of elliptic curves with $H$-level structure, with $\deg(\alpha) = \ell$. Then we have that $[\alpha \phi] = [\phi']$. Thus there exists $h \in H$ such that $\alpha \circ \phi = \phi' \circ h$, and so $[\ell] \circ \phi \circ h^{-1} = \hat{\alpha} \circ \phi'$. It follows that $\hat{\alpha} : (E', [\phi']) \to (E, [\ell \phi])$ is a morphism of elliptic curves with $H$-level structure. 

We will use the dual map to define the map $J$ for $G(p,\ell, H)$ as an abstract isogeny graph. To start, let $\widetilde{Y}_{ij}$ denote the set of degree $\ell$ morphisms of $H$-level structures from $(E_i, [\phi_i])$ to $(E_j,[\phi_j])$. Note that $\Aut(E_j, [\phi_j])$ and $\Aut(E_i, [\phi_i])$ give compatible left and right actions on $\widetilde{Y}_{ij}$, defined by post-composition and pre-composition, respectively.

\begin{proposition}\label{prop:well-defined-pre-J}
Let $\{x_i = (E_i, [\phi_i])\}$ be a complete set of representatives of the isomorphism classes of supersingular elliptic curves with $H$-level structure, and $\widetilde{Y}_{ij}$ be the set of degree $\ell$ morphisms of $H$-level structures from $x_i$ to $x_j$. Let $\mathcal{Y}\coloneqq \bigsqcup_{i,j} \Aut(x_j)\backslash \widetilde{Y}_{ij} /\Aut(x_i)$, and $\tilde{J} : \mathcal{Y} \to \mathcal{Y}$ be the function
\begin{align*}
 \tilde{J}\colon \Aut (E_j,[\phi_j])\backslash \widetilde{Y}_{ij}/\Aut(E_i,[\phi_i]) &\to \Aut(E_k,[\phi_k])\backslash \widetilde{Y}_{jk} /\Aut(E_j,[\phi_j]) \\ 
\Aut(E_j,[\phi_j]) \alpha\Aut(E_i,[\phi_i]) &\mapsto \Aut(E_k,[\phi_k]) (u\circ\widehat{\alpha}) \Aut(E_j,[\phi_j]),
\end{align*}
where $u$ is any choice of isomorphism $(E_i, [\ell \phi_i]) \to (E_k, [\phi_k])$. Then $\tilde{J}$ is well-defined and independent of the choice of $u \in \operatorname{Isom}((E_i, [\ell \phi_i]), (E_k, [\phi_k]))$.
\end{proposition}

\begin{proof}
First, we note that $\Aut(E_k, [\phi_k])  (u \circ \hat{\alpha}) \Aut(E_j, [\phi_j])$ is independent of the choice of $u$, since for any two isomorphisms $u, v : (E_i, [\ell \phi_i]) \to (E_k, [\phi_k])$, $vu^{-1} \in \Aut(E_k, [\phi_k])$, and $v = (vu^{-1})u$.

We now argue that $\tilde{J}$ is well-defined, so we must show that $\tilde{J}$ does not depend on the representative of the orbit $\Aut(E_j, [\phi_j]) \backslash \widetilde{Y}_{ij} / \Aut(E_i, [\phi_i]).$ Suppose that $\alpha_1, \alpha_2$ are two representatives for this orbit. Then we have automorphsims $u_i \in \Aut(E_i, [\phi_i])$ and $u_j \in \Aut(E_j, [\phi_j])$ such that $u_j \circ \alpha_1 \circ u_i = \alpha_2$. Note that $\Aut(E, [\phi]) = \Aut(E, [\ell \phi])$, since $u \in \Aut(E, [\phi])$ implies that $u_i \circ \phi_i = \phi_i \circ h$, and so $\ell \circ u_i \circ \phi_i = u_i \circ \ell \phi_i = \ell \phi_i \circ h$, so $u \in \Aut(E, [\ell \phi])$. The reverse containment follows from the fact that the map $[\ell]$ restricted to $E[N]$ is bijective. Thus we have that $u_i^{-1} \circ \hat{\alpha}_1 \circ u_j^{-1} = \hat{\alpha}_2$, where $u_i^{-1} \in \Aut(E_i, [\ell \phi_i])$. For any $u \in \Aut(E_k, [\phi_k]),$ we have that $\Aut(E_k, [\phi_k])u \hat{\alpha}_2 \Aut(E_j, [\phi_j]) = \Aut(E_k, [\phi_k])u u_i^{-1}\hat{\alpha}_1 u_j^{-1} \Aut(E_j, [\phi_j])$, and since $u u_i^{-1}$ is an isomorphism from $(E_i, [\ell \phi_i])$ to $(E_k, [\phi]_k)$, the fact that $\tilde{J}$ is well-defined follows from the fact that the orbit is independent of the choice of isomorphism $u$.
\end{proof}

Using the function $\tilde{J}$ of Proposition \ref{prop:well-defined-pre-J}, we define $G(p,\ell, H)$ as an abstract isogeny graph in Definition \ref{def:GplH-abstract-iso-graph}.

\begin{definition}[The $\ell$-isogeny graph of elliptic curves with $H$-level structure]\label{def:GplH-abstract-iso-graph}
    Let $N$ be a positive integer, let $p$ and $\ell$ be distinct primes coprime to $N$, and let $H$ be a subgroup of $\GL_2(\ZZ/N\ZZ)$. The {\bf $\ell$-isogeny graph of supersingular elliptic curves with $H$-level structure}, denoted $G(p, \ell, H) = (X, Y, s, t, J, L)$, is the abstract isogeny graph with
    \begin{description}
        \item[Vertices] 
        A complete set of representatives $x_1=(E_1,[\phi_1]),\ldots,x_n=(E_n,[\phi_n])$ for isomorphism class of supersingular elliptic curves with $H$-level structure;
        \item[Edges] Given vertices $x_i=(E_i,[\phi_i])$ and $x_j=(E_j,[\phi_j])$, let $\widetilde{Y}_{ij}$ denote the degree $\ell$ morphisms $x_i\to x_j$. Edges from $x_i$ to $x_j$ are orbits in $\widetilde{Y}_{ij}$ of the left action of $\Aut(x_j)$. In particular, an edge is of the form  $\Aut(E_j,[\phi_j])\alpha$ where $\alpha\colon (E_i,[\phi_i])\to (E_j,[\phi_j])$ is a  degree $\ell$ morphism.
        \item[Source and target] The source and target of $y=\Aut(x_j)\alpha$ are $x_i$ and $x_j$, respectively, where $\alpha\in \widetilde{Y}_{ij}$. 
        \item[J] For each orbit $O=\Aut(x_j)\alpha\Aut(x_i)$, pick a representative morphism $\alpha_O$. For each edge of the form $y=\Aut(x_j)\alpha_Ou$ with $u\in \Aut(x_i)$, define $Jy=\Aut(Lx_i)\alpha_{\tilde{J}(O)}$. 
        \item[$\mathbf{L}$] Define $L(E_i,[\phi_i])\coloneqq x_k$, where $x_k\in X$ is isomorphic to  $(E_i,[\ell\phi_i])$.
    \end{description}
\end{definition}
\begin{proposition}
    $G(p,\ell,H)$ (with the above definition of $s,t,J,L$)  is an $(\ell + 1)$-regular abstract isogeny graph. 
\end{proposition}
\begin{proof}
    We have to show that $s(Jy)=t(y)$ and $t(Jy)=Ls(y)$ for all edges $y$. Let $y$ be an edge with $s(y)=x_i$ and $t(y)=x_j$ and write $y=\Aut(x_j)\alpha = \Aut(x_j)\phi_O u$ for some $u\in \Aut(E_i)$ where $O=\Aut(x_i)\alpha\Aut(x_j)$. Then $Jy=\Aut(x_k)\phi_{\widehat{O}}$, and $s(Jy)$ is the domain of $\alpha_{\widehat{O}}\colon x_j\to x_k$, i.e. $x_j=t(y)$. And $t(Jy)$ is $x_k=(E_i,[\ell\phi_i]) = Lx_i = Ls(y)$.  The graph is $\ell+1$-regular since each curve has $\ell+1$ distinct cyclic subgroups of its $\ell$-torsion. 
\end{proof}

\begin{remark}\label{rmk:backtracking-reconcilliation}
Recall from \cite{CyclesInGpl, OrientationsAndCycles, Codogni-Lido} that a cycle in $G(p,\ell, H)$ is typically considered to have backtracking if, after choosing a set of representatives for the edges, there are consecutive edges that compose to $u[\ell]$, where $u$ is an automorphism. From the construction of the $J$ function above, and Definition \ref{def:abstract-iso-graph-defs}, we see that our definition of backtracking in $G(p,\ell,H)$ as an abstract isogeny graph agrees with the definition used in previous work, for the particular choice of representatives we have given. This choice of representatives is further justified by the fact that it is a \emph{safe arbitrary assignment} in the sense of \cite[Definition 3.13]{OrientationsAndCycles}    
\end{remark}

We now consider the orientable graphs associated to $G(p,\ell, H)$ in the case $\ell \in H$. Let $\Gamma=G(p,\ell,H)$ and suppose $\ell\in H$. Then $L\colon X\to X$ is the identity, and we have $s(Jy)=t(y)$ and $t(Jy)=s(y)$ for all edges $y$ of $\Gamma$. We can construct $\Gamma^{\pm1}$ directly. Let the vertex set be $X$, a complete set of representatives of isomorphism classes of supersingular elliptic curves with $H$-level structure. Define
\[
\mathcal{Y}\coloneqq \bigsqcup_{i,j} \Aut(x_j)\backslash \widetilde{Y}_{ij} /\Aut(x_i).
\]
The edges of $\Gamma^{+1}$ are 
\[
Y^{+1}= \mathcal{Y} -\{O\in \mathcal{Y}:\widehat{O}=O\}
\]
and the edges of $\Gamma^{-1}$ are 
\[
Y^{-1}= \mathcal{Y}\sqcup \{O\in \mathcal{Y}:\widehat{O}=O\}.
\]

\begin{proposition}\label{prop:ell-in-H-implies-J-involution}
Let $\Gamma = G(p,\ell, H)$ and suppose that $\ell \in H$. Then the permutation induced by $J$ is an involution, and  
    \[
    \zeta_{\Gamma}(u) = \frac{(1-u)^{\chi(\Gamma^{+1})}(1+u)^{\chi(\Gamma^{-1})}}{\det(1-uA+\ell u^2)}.
    \]
\end{proposition}

\begin{proof}
Let $O = \Aut(x_j)\alpha \Aut(x_i) \in \mathcal{Y}$, and $(E_i, [\phi_i])$, $(E_j, [\phi_j])$ be representatives for $x_i, x_j$ respectively. Then we have $t(J(O)) = Ls(O) = Lx_i = (E_i, [\ell \phi]).$ Since $\ell \in H$, we have that $(E_i, [\ell \phi]) = (E_i, [\phi]) = x_i$. Thus $J(O) = \Aut(x_i)\hat{\alpha}\Aut(x_j)$, and $J^2(O) = O$. The claim about $\zeta_\Gamma(u)$ now follows from Corollary \ref{cor:Ihara-determinant-J-involution}.
\end{proof}

In Appendix \ref{app:ExplicitEulerChars} we give complete formulas in Theorem~\ref{thm:LevelStructureEulerCharacteristic-GeneralCase} for $\chi(G(p,\ell,B_0(N)^{\pm 1})$ and a formula for the zeta function of $G(p,\ell,H)$ when $p\equiv 1\pmod{12}$ and $(-\ell|p)=1$ in Theorem~\ref{thm:Ihara_Zeta_permutations_version}.

\begin{remark}
    There is more structure available in the data of an isogeny graph with level structure than is captured by our definition of an abstract isogeny graph. Each vertex $(E,[\phi])$ in $G(p,\ell,H)$ is equipped with an automorphism group $\Aut (E,[\phi])$ that acts on the edges whose source is $(E,[\phi])$. Each edge $\Aut(E',[\phi'])\alpha$ with source $(E,[\phi])$ can be assigned a group in a natural way as its stabilizer under the action of $\Aut(E,[\phi])$, and we have a natural injection  $\Stab_{\Aut(E,[\phi])}(\Aut(E',[\phi'])\alpha)\hookrightarrow \Aut(E,[\phi])$. Thus one could generalize the notion of a {\bf graph of groups}, as in~\cite[\S4.4, Definition 8]{trees}, to the setting of abstract isogeny graphs; we leave this for future work.
\end{remark}

\subsection{$(\ell,\ldots,\ell)$-isogeny graphs}
The $(\ell,\ldots,\ell)$-isogeny graph of superspecial principally polarized abelian varieties of dimension $g$ can also be seen as an abstract isogeny graph. For more details, we refer the reader to the work of Jordan--Zaytman~\cite{jordan2023isogenygraphssuperspecialabelian} and Florit--Smith~\cite{Florit-Smith}. Define a graph $\Gamma=\Gamma(g,p,\ell)$ whose vertex set $X$ is a complete set of representatives of the isomorphism classes of principally polarized superspecial abelian varieties $A_1,\ldots,A_h$. The edge set $Y$ consist of  equivalence classes of $(\ell,\ldots,\ell)$-isogenies $A_i\to A_j$; these equivalence classes correspond to Lagrangian subgroups $C\subseteq A_i[\ell]$ such that $A_i/C\simeq A_j$, where $C$ is Lagrangian if it is maximally isotropic with respect to the Weil pairing. As before, we call two isogenies equivalent if they differ by post-composition by automorphism. This is the ``big isogeny graph'' in the language of Jordan and Zaytman~\cite{jordan2023isogenygraphssuperspecialabelian}. The source and target of an edge are the representative isogeny's domain and codomain. And we can define $J\colon Y\to Y$ as before, by choosing a representative $\alpha_O$ for each orbit $O=\Aut(A_j)\alpha\Aut(A_i)$ and then defining $J(\Aut(A_j)\alpha) = \Aut(A_i)\alpha_{\widehat{O}}$ where $O=\Aut(A_j)\alpha \Aut(A_i)$. The map $L\colon X\to X$ is the identity.

\section{Relation to zeta functions of modular curves}\label{sec:RelationToModularCurves}

We fix two distinct primes $p$ and $\ell$, an integer $N>0$ with $\gcd(p\ell,N) =1$, and a subgroup $H$ of $\GL_2(\ZZ/N\ZZ)$. Throughout this section, we denote $G(p,\ell, H)$ by $G$. The goal of this section is to relate the Ihara zeta function of $G$ (see Definition~\ref{def:ihara-zeta-for-abstract-iso-graphs}) to the Hasse-Weil zeta function associated to the modular curves $X_{H,\FF_{\ell}}\coloneqq (X_H)_{\FF_{\ell}}$ and $X_{H_p,\FF_{\ell}}\coloneqq (X_{H_p})_{\FF_{\ell}}$, where $H_p = H \times B_0(p) \subset GL_2(\ZZ/pN\ZZ)$.

\subsection{Weil invariants and a determinant computation}

We briefly summarize the definitions and results of \cite{Codogni-Lido} that we will use in our proof.  

Let $\mu_N^\times(\Fpbar)$ denote the set of primitive $N$-th roots of unity in $\Fpbar$. There is a right action of $(\ZZ/N\ZZ)^\times$ on $\mu_N^\times(\Fpbar)$ via $\zeta \cdot a = \zeta^a$, thus inducing an action of $\det(H)$ on $\mu_N^\times(\Fpbar)$. Denote the orbits of this action by $R_H := \mu_N^\times(\Fpbar)/\det(H)$. For any level-$N$ structure $\phi\colon (\ZZ/N\ZZ)^2\to E[N]$, we have that $(\phi(1,0),\phi(0,1))$ form a basis for $E[N]$, and since the Weil pairing $e_N$ is bilinear and alternating we have for any $h\in H$ that 
\[
e_N((\phi\circ h)(1,0),(\phi\circ h)(0,1)) = e_N(\phi(1,0),\phi(0,1))^{\det h}.
\]
By the definition of the Weil pairing, we have $e_N(\phi(1,0),\phi(0,1))\in \mu_N^{\times}$, so combining this with the previous comment, $w(E,[\phi]):= e_N(\phi(1,0), \phi(0,1))\in R_H$ is well-defined. We give the following definition of the Weil invariant of the level structure, as in \cite[Definition 1.5]{Codogni-Lido}.

\begin{definition}[Weil invariant of the level structure]\label{def:weilinvariant}
    Let $(E,[\phi])$ be an elliptic curve with level structure. Let $e_N$ be the Weil pairing on $E[N]$. We define $w(E,[\phi]):=e_N(\phi(1,0), \phi(0,1)) \in R_{H}$ to be the {\bf Weil invariant} of the level structure.
\end{definition}

The Weil invariant gives a surjective map of graphs from $G$ to the directed Cayley graph $C=C(N, \det(H), \ell)$. The vertices $\zeta_i$ of $C$  are elements of $R_H$, and there is an edge from $\zeta_i$ to $\zeta_j$ if $\zeta_j = \zeta_i^{\ell}$. The function $w: V(G) \to R_H$ induces a surjective linear map $w_*: \CC^{V(G)} \to \CC^{R_H}.$ 

For each connected component $C_i$ of $C$, let $G_i = w^{-1}(C_i)$. Then $w_i: G_i \to C_i$ is a surjective map of graphs, and there is a corresponding linear map $w_{i,*}: \CC^{V(G_i)} \to \CC^{V(C_i)}.$ Write $V_\xi$ as the set of vertices in $G$ with Weil invariant $\xi \in R_H$. Then we have

\[
\ker(w_*) = \bigoplus_{\xi \in R_H} \{ (x_v) \in \CC^{V_\xi} : \sum x_v = 0\}\ \ , \ \  \ker(w_{i,*}) = \bigoplus_{\xi \in V(C_i)} \{ (x_v) \in \CC^{V_\xi} : \sum x_v = 0\}.
\]

\begin{definition}\label{def:innerproduct}
Let $e_i=\# \Aut(E_i, \phi_i)$ and 
    define an inner product on $\CC^{V(G)}$ by 
    \[
    \langle (E_i, \phi_i), (E_j, \phi_j) \rangle = \begin{cases}
        e_i &: i=j \\ 
        0 &: i\not=j.
    \end{cases}
    \]
\end{definition}

Next, we reduce the computation of $\zeta_G(u)$ to the computation of $\det(\id_{\CC^{V(G)}} -uA+ \ell L u^2| \ker(w_*) )$.

\begin{proposition}\label{prop:keromega} Let $k$ be the order of $\ell$ in $(\ZZ/N\ZZ)^{\times}/\det(H)$, and let $n = \frac{\phi(N)}{k |\det(H)|}.$ 
Let $A$ be the adjacency matrix of $G$, let $I=\id_{\CC^{V(G)}}$, and let $Q = D-I$ as in Section \ref{sec:IharaDeterminantFormula}.
Then \[\det(I -uA+ u^2QL|\CC^{V(G)}) = (1 - \ell^k u^k)^n(1 - u^k)^{n}\prod_{i=1}^{n}(\det(I -uA+ u^2 \ell L | \ker(w_{i,*})).\]
\end{proposition}

\begin{proof}
    Because all vertices of $G$ have degree $\ell + 1$, we have $Q = \ell \id_{\CC^{V(G)}}$.  
    Since $n$ is the number of connected components of $C=C(N, \det(H), \ell)$, it suffices to fix a connected component $C_i$ and compute $\det(\id_{\CC^{V(G)}} - uA + u^2QL| \CC^{V(G_i)})$ component by component.
    
    Note that $\ker(w_{i, *})$ is invariant under $A$ and $L$: Elements of $\ker(w_{i,*})$ are of the form $\sum a_{(E,[\phi])} (E, [\phi])$ such that for each $\zeta \in \vertices C_i$, $\sum_{w(E, \phi) = \zeta} a_{(E, \phi)}  = 0$. Since $A(E, [\phi])$ is a sum of $(\ell + 1)$ vertices $(E',[\phi'])$ whose Weil invariant is equal to $w(E, [\phi])^{\ell}$, this property is preserved under $A$. Similarly, $w(L(E, [\phi]))=w(E, [\ell]\phi) = w(E, \phi)^{\ell^2}$, so this property is preserved under $L$ as well.

    Let $\zeta_1, \zeta_2, \ldots, \zeta_k$ denote the vertices of $C_i.$ For each $\zeta_j$, let \[\delta_j = \sum_{w(E, [\phi]) = \zeta_j} \frac{1}{\#\Aut(E, [\phi])}(E, [\phi]).\] Each $\delta_j$ is orthogonal to $\ker(w_{i,*})$ under the inner product defined in Definition \ref{def:innerproduct}, and $\{\delta_j\}_{j=1}^{k}$ are linearly independent. Let $U_i$ be the span of the $\{\delta_j\}_{j=1}^{k}$. Since $k=\dim(U_i) = \dim(\CC^{V(C_i)}) = \dim(\CC^{V(G_i)}) - \dim(\ker(w_{i,*}))$ by rank-nullity, $U_i$ is equal to the orthogonal complement of $\ker(w_{i,*})$ in $\CC^{V(G_i)}$. 

    Fix $\delta_j$. We will compute $A\delta_j$ and $L \delta_j$.
    First, $A\delta_j = (\ell + 1)\delta_{j'}$, where $\zeta_{j'} = \zeta_j^{\ell}$. To see this, we compute: 
\begin{align*}
        A \delta_j &= \sum_{w(E, [\phi]) = \zeta_j} \sum_{(E', [\phi']) \in V(G_i)} \frac{\#\{\text{edges: } (E, [\phi]) \to (E', [\phi'])\}}{\#\Aut(E, [\phi])} (E',[\phi'])\\
        &= \sum_{w(E', [\phi']) = \zeta_j^{\ell}} (E', [\phi']) \sum_{(E, [\phi]) \in V(G_i)} \frac{\#\{\text{edges: } (E, [\phi]) \to (E', [\phi'])\}}{\#\Aut(E, [\phi])}
    \end{align*}
As $\frac{\#\{\text{edges: } (E, \phi) \to (E', \phi')\}}{\#\Aut(E, \phi)} = \frac{\#\{\text{edges: } (E', \phi') \to (E, [\ell] \phi)\}}{\#\Aut(E', \phi')}$ (see the proof of \cite[Proposition 2.2]{Codogni-Lido}), the terms can be rewritten as
      \[A \delta_j =\sum_{\omega(E', [\phi']) = \zeta_j^{\ell}} \frac{(E', [\phi'])}{\#\Aut(E', [\phi'])} \sum_{(E, [\phi]) \in V(G_i)} \#\{\text{edges: } (E', [\phi']) \to (E, [\ell] \phi)\}.\]  
      
      Note that the inner sum is over all $(E, [\phi])$ with an isogeny to $(E',[\phi'])$, as all such $(E,[\phi])$ must have the same Weil invariant $\zeta_j$. Therefore, the edges $(E', [\phi']) \to (E, [\ell \phi])$ are all outgoing edges from $(E', [\phi'])$, and we obtain
      \[A \delta_j = (\ell + 1)\sum_{w(E', [\phi']) = \zeta_j^{\ell}} \frac{(E', [\phi'])}{\#\Aut(E', [\phi'])} = (\ell + 1) \delta_{j'}.
      \]

    Next, we show $L\delta_j = \delta_{j''}$, where $\zeta_{j''} = \zeta_j^{\ell^2}$. To see this, note that $L \delta_j= \sum_{w(E,[\phi]) = \zeta_j} \frac{(E, [\ell\phi])}{\#\Aut(E, [\phi])}.$ Since $\#\Aut(E,[\phi]) = \#\Aut(E, [\ell\phi])$, this is $L \delta_j = \sum_{w(E, [\phi]) = \zeta_j} \frac{(E, [\ell\phi])}{\#\Aut(E, [\ell\phi])}.$ Now it suffices to show that the vertices $(E, [\ell\phi])$ range over all vertices with $w((E',[\phi'])) = \zeta_j^{\ell^2}$. Choose $m$ such that $\ell^m \in H$, and note that $k \mid 2m$. Then for any $(E', [\phi'])$ with $w(E', [\phi']) = \zeta_j^{\ell^2}$, we have $w((E', [\ell^{m-1} \phi'])) = \zeta_j^{\ell^{2m}}.$ Since $k \mid 2m,$ $\zeta_j^{\ell^{2m}} = \zeta_j$. Since $\ell^m \in H$, it follows that $(E', [\phi']) = (E', [\ell^m \phi']) = (E', [\ell][\ell^{m-1}\phi'])$, which shows that $(E', [\phi']) = (E, [\ell\phi])$ for some $(E, [\phi])$ with $w(E, [\phi]) = \zeta_j$. \\

    \textbf{Case 1.} $k=1$.

    In this case, $\zeta_j=\zeta_j^{\ell}=\zeta_j^{\ell^2}$, so the restriction of $\id_{\CC^{V(G)}} - uA + u^2QL$ to the single vertex of $U_i$ is $(1 - (\ell+1)u + \ell u^2) = (1 - \ell u)(1 - u).$\\

    \textbf{Case 2.} $k = 2$.

    If $k = 2$, then $U_i$ has two distinct vertices $\zeta_j$ and $\zeta_j^{\ell}$, and $\zeta_{j}^{\ell^2} = \zeta_j$.

    In this case, the restriction of $\id_{\CC^{V(G)}} -uA+ u^2Q L$ to $U_i$ has the form \[\begin{pmatrix}
        (1 +\ell u^2) & -(\ell + 1)u \\
        -(\ell + 1)u & (1 + \ell u^2)
    \end{pmatrix},\] whose determinant is $(1 + \ell u^2)^2 - (-(\ell + 1)u)^2 = (1-\ell^2 u^2)(1 - u^2).$\\

    \textbf{Case 3.} $k \geq 3$.

    In this case, $\zeta_j$, $\zeta_j^{\ell}$, and $\zeta_j^{\ell^2}$ are all distinct. Therefore, the restriction of $(\id_{\CC^{V(G)}} - uA + u^2QL)$ to $U_i$ has the form of the circulant matrix whose first row is $1, -(\ell+1)u, \ell u^2$ followed by $k-3$ zeroes:
    
    \[\begin{pmatrix}
        1 & -(\ell + 1)u & \ell u^2 &\ldots\\
        0 & 1& -(\ell + 1)u  & \ell u^2 & \ldots\\
        \vdots & \vdots  & \ddots & \ddots  & \ddots \\
        -(\ell + 1)u & \ell u^2 & \ldots & 0 &1\\
    \end{pmatrix},\] whose determinant is  $\prod_{j=1}^{k} (1 -(\ell+1)u \omega^j + \ell u^2 \omega^{2j}),$ where $\omega$ is a primitive $k$-th root of unity. Factoring this expression as $\prod_{j=1}^k (1 - \ell u \omega^j)(1 - u \omega^j)$ and noting $\omega^j$ ranges over all k-th roots of 1, we obtain $\det(1 - Au + QLu^2|U_i) = (1 - \ell^k u^k)(1 - u^k).$
\end{proof}

\subsection{Product formulas}

Next, we will show how to relate the Ihara zeta function of $G$ to the Hasse-Weil zeta functions of certain modular curves. 
We use the definition of a modular curve from~\cite[Section 2.4]{RSZB}; see also~\cite[VI.3]{DeligneRapoport}.
\begin{definition}\label{def:modcurves}
For a subgroup $H \subset \GL_2(\ZZ/N\ZZ)$,
define $X_{H}$ to be the coarse moduli space of the algebraic stack
$\mathcal{M}_{H}$ that parameterizes generalized elliptic curves with $H$-level structure. 
\end{definition}
We also denote $H_p = H \times B_0(p) \subset \GL_2(\ZZ/N\ZZ)\times \GL_2(\ZZ/p\ZZ) \cong \GL_2(\ZZ/pN\ZZ)$. 
We will consider the modular curves $X_H$ and $X_{H_p}$, which are smooth projective curves with canonical models defined over $\QQ$ (in fact, over $\ZZ[1/N]$ and $\ZZ[1/pN]$ respectively). We will typically restrict to the case that $\det(H) = (\ZZ/N\ZZ)^\times$ (which implies $\det(H_p) = (\ZZ/pN\ZZ)^{\times}$ as well). In this case, $X_{H}$ is irreducible and isomorphic over $\CC$ to the complex Riemann surface obtained from the quotient of $\mathcal{H}^* := \mathcal{H} \cup \PP^1(\QQ)$ by the action of $\Gamma_{H} := \{ \gamma \in \SL_2(\ZZ) : \gamma^T \pmod{N} \in H \}$.


\begin{definition}
    Let $X$ be a smooth projective variety defined over $\FF_\ell$. The \textbf{Hasse-Weil zeta function} for $X$ is defined as: \[Z(X, u)= \exp(\sum_{n=1}^{\infty} \frac{\#X(\FF_{\ell^n})}{n} u^{n}) = \prod_{x \in [X]} \frac{1}{1-u^{\deg(x)}},\] where the product is defined over the closed points of $X$.
\end{definition}

One can deduce the following proposition from \cite[Theorem 7.11]{Shimura-arithmetic-theory-automorphic}, but we include a proof in the language of $H$-level structures.

\begin{proposition}
    Fix a prime $\ell \nmid M$ and a subgroup $H' \leq \GL_2(\ZZ/M\ZZ)$ such that $B_1(M) \subset H' \subset B_0(M).$ Let $m_{\ell}\in \GL_2(\ZZ/M\ZZ)$ satisfy $m_\ell \equiv \begin{psmallmatrix}
    \ell^{-1} & 0 \\ 0 & \ell
\end{psmallmatrix} \pmod{M}$
and define 
\[
    \widetilde{T}_{\ell} = [\Gamma_{H'} \begin{psmallmatrix}
        1 & 0 \\
        0 & \ell
    \end{psmallmatrix} \Gamma_{H'}], \quad \langle \tilde{\ell}\rangle = [\Gamma_{H'} m_{\ell}\Gamma_{H'}].
\]

 Then the Hasse-Weil zeta function of the modular curve $X_{H',\FF_{\ell}}$ is given by
\[
    Z(X_{H', \FF_{\ell}}, u) = \frac{\det(1 - \widetilde{T}_{\ell}u - \langle \widetilde{\ell}\rangle \ell u^2|S_2(\Gamma_{H'}))}{(1 - u)(1 - \ell u)}
\]

\end{proposition}

\begin{proof}   
    Fix a prime $q \nmid \ell M$.
    Using the Grothendieck-Lefschetz trace formula, the Hasse-Weil zeta function can be written in terms of the action of the $\ell$-power Frobenius $F$ on $q$-adic cohomology groups. For a projective curve $X$ defined over $\FF_{\ell}$, write $P_i(X, u) = \det(1 - uF|H^i_{et}(X, \QQ_{q}))$. Then,
    \[Z(X_{H', \FF_{\ell}}, u) = \frac{P_1(X_{H', \FF_{\ell}}, u)}{P_0(X_{H', \FF_{\ell}}, u) P_2(X_{H', \FF_{\ell}}, u)}.\]

    Since $\det(H') = (\ZZ/M\ZZ)^{\times}$, the modular curve $X_{H'}$ is geometrically irreducible, and since $\ell$ does not divide $M$, its base change to $\FF_{\ell}$, denoted  $X_{H',\FF_{\ell}}$, is geometrically irreducible. Therefore, $P_0(X_{H',\FF_{\ell}}, u) = 1 - u$, as there is one connected component and $F$ acts as the identity on $H^0_{et}(X_{H',\FF_{\ell}}, \QQ_{q})$. By duality of $H^0$ and $H^2$, we have $P_2(X_{H',\FF_{\ell}}, u) = 1 - \ell u$. 

    To write $P_1(X_{H',\FF_{\ell}}, u)$ in terms of the action of $\widetilde{T}_{\ell}$ on $S_2(\Gamma_{H'})$, we will first relate the action of F on $H^1_{et}(X_{H',\FF_{\ell}}, \QQ_q)$ to the action of $F$ on $\Pic^0(X_{H',\FF_{\ell}})$. 
    Let $\mathbb{T}$ denote the $\QQ_q$-algebra generated by Hecke operators $T_r$ with $r \nmid N$ and diamond operators $\langle d\rangle$, acting on $\Pic^0(X_{H',\FF_{\ell}})$. 
     Viewing $X_{H'}$ as the compact Riemann surface obtained from $\Gamma_{H'} \backslash \mathcal{H}^*$, there are isomorphisms \[H^1((X_{H'})^{an}, \QQ_q) \cong \text{Tate}_{q}(\Pic^0(X_{H'})) \otimes \QQ_{q} \cong \text{Tate}_q(\Pic^0(X_{H',\FF_{\ell}} ) \otimes \QQ_{q}.\] Let $V_{q} = \text{Tate}_{q}(\Pic^0(X_{H',\FF_{\ell}}) \otimes \QQ_{q}$, and note that these are isomorphisms as $\mathbb{T}$-modules.
    
    On $V_q$, the Eichler-Shimura relation is given by \cite[Theorem 3.7]{Codogni-Lido} \[T_{\ell, \FF_{\ell}} = F + \langle \ell\rangle F^{\vee}.\] Composing with $F$, we get the equation \[T_{\ell} F =F^2 + \langle \ell\rangle \ell.\]
    Therefore, $F$ satisfies the quadratic equation, \[x^2 - T_{\ell} x + \langle \ell \rangle \ell = 0.\] To see that this is the characteristic polynomial of F on $V_q$, we show that $\Tr^{V_q}_{\mathbb{T}}(F) = T_{\ell}.$ In particular, we show that $\Tr^{V_q}_{\mathbb{T}}(F) = \Tr^{V_q}_{\mathbb{T}}(\langle \ell \rangle F^{\vee})$.
    
    There is a Hecke-equivariant isomorphism from $V_q$ to its dual, $V_q^{*} \cong \Hom_{\QQ_q}(V_q, \QQ_{q})$, which is induced by the following map:
Fix $\zeta \in \mu_{M}^{\times}$. Let $(E, [\phi])$ be a point on $X_{H',\FF_{\ell}}$. Since the determinant of $H'$ is surjective, we may choose a representative $\phi$ such that $e_M(\phi(1,0), \phi(0,1)) = \zeta$. Let $P = \phi(1,0)$ and $Q = \phi(0,1)$.
    Let $\alpha: E \to E'$ be the isogeny with $\ker(\alpha) = \langle P\rangle$. Finally, choose any $H'$-level structure $\phi_{\alpha(Q)}$ on $E'$ such that $\phi_{\alpha(Q)}(1,0) = \alpha(Q)$ and $w((E, [\phi_{\alpha(Q)}]) = \zeta$. We define \[f_{\zeta}(E, [\phi]) = (E', [\phi_{\alpha(Q)}]).\]

    To show that $f_{\zeta}$ is well-defined, we need to show that $f_{\zeta}$ does not depend on the representative $\phi$. Suppose $\phi'$ is an equivalent level structure on $E$ such that $e_M(\phi'(1,0), \phi'(0,1)) = \zeta$.
   Then $\phi \circ h = \phi'$ for some $h \in H'$. As $H' \subset B_0(M)$, this implies $h$ is of the form $\begin{pmatrix}
        a & 0\\
        c &d
    \end{pmatrix}$, so that $\phi'(1,0) = a \phi(1,0) = aP$ and $\phi'(0,1) = c \phi(1,0) + d \phi(0,1) = cP + dQ$. Since the Weil pairing is bilinear and alternating, we have $\zeta =e_M(aP, cP + dQ) = e_M(aP, cP) e_M(aP, dQ) = e_M(P,P)^{ac} e_M(P,Q)^{ad} = e_M(P,Q)^{ad} = \zeta^{ad}.$ Therefore $ad \equiv 1 \pmod{M}$, and in particular, $P$ and $aP$ generate the same subgroup of $E$. Thus, $\alpha: E\to E'$ is well-defined up to isomorphism and post-composition by an automorphism. 

    Let $\phi'_{\alpha(Q')}$ denote any level structure on $E'$ such that 
    \[\phi'_{\alpha(Q')}(1,0) = \alpha(\phi'(0,1)),\quad e_M((1,0), (0,1)) = \zeta.\] Since $\alpha(\phi'(0,1)) = \alpha(cP + dQ) = d \cdot \alpha(Q),$ we also have $e_M(\alpha(Q), \phi_{\alpha(Q)}(0,1)) = \zeta = e_M(d\cdot\alpha(Q), \phi'_{\alpha(Q')}(0,1)$. By a similar argument as above, and using $ad \equiv 1 \pmod{M}$, we have that $\phi_{\alpha(Q)} = \phi'_{\alpha(Q')} \circ h',$ where $h' = \begin{pmatrix} d & 0\\ *  & a \end{pmatrix}$. To see that $h' \in H$, note that $h' = h^{-1} \cdot \begin{pmatrix}
        1 & 0 \\
        a^{-1}(* + c) & 1
    \end{pmatrix}$ and $H \supset B_1(M).$ This shows that $\phi_{\alpha(Q)}$ and $\phi'_{\alpha(Q')}$ are equivalent as $H'$-level structures; note that setting $d=1$ shows that any two $H'$-level structures mapping $(1,0)$ to $\alpha(Q)$ are equivalent.

    Then $f_{\zeta}$ induces an isomorphism $V_q \to \Hom_{\QQ_q}(V_q, \QQ_q)$, via a modification of the typical Weil pairing $e_M$ on $\Pic^0(X_{H', \FF_{\ell}})$: \[x \mapsto \Big(y \mapsto [x,y]:=e_M( x, f_{\zeta} y) \Big).\]

    For all $T \in \mathbb{T}$, we have $[Tx,y] = [x, Ty]$, and $[f_{\zeta} x, y] = [x, f_{\zeta} y]$. As the adjoint of $F^{\vee}$ under the Weil pairing is its dual $F$, the adjoint of $F^{\vee}$ under the modified Weil pairing is $f_{\zeta}^{-1} F f_{\zeta}$. We show that $f_{\zeta}^{-1}F f_{\zeta} = \langle \ell \rangle^{-1} F$.

    Equivalently, we need to show $F f_{\zeta} = f_{\zeta} \langle \ell\rangle^{-1} F$ on $\Pic^0(X_{H',\FF_{\ell}})$. Fix $\phi$ a level structure on $E$ such that $e_M(\phi(1,0), \phi(0,1)) = \zeta$. Let $\phi(1,0)=P$ and $\phi(0,1) = Q$. We evaluate each map on $(E, [\phi])$:
    \begin{align*}
        F f_{\zeta} (E, [\phi]) = F(E', [\phi_{\alpha(Q)}]) = (E'^{(\ell)}, [F \circ \phi_{\alpha(Q)}])\\
        f_{\zeta} \langle \ell\rangle^{-1} F(E, [\phi]) = f_{\zeta} (E^{(\ell)}, [\ell ^{-1} F \circ \phi])
    \end{align*} 
    Here, $\alpha: E \to E'$ is the isogeny with kernel $P$ and $\phi_{\alpha(Q)}(1,0) = \alpha(Q)$.

    To apply $f_{\zeta}$ to $(E^{(\ell)}, [\ell^{-1} F \circ \phi])$, we first need to choose a representative $\phi'$ for $[\ell^{-1} F \circ \phi]$ such that $e_M(\phi'(1,0), \phi'(0,1)) = e_M(P,Q) = \zeta$. We have $w(E^{(\ell)}, [\ell^{-1} F \circ \phi]) = e_M(\ell^{-1}P^{(\ell)}, \ell^{-1} Q^{(\ell)})$. By linearity and Galois invariance of the Weil pairing, this is $e_M(P,Q)^{\ell^{-1}} = \zeta^{\ell^{-1}}$. Note that $\begin{pmatrix}
        1 & 0\\
        0 & \ell 
    \end{pmatrix} \in B_1(M) \subset H'$, so $\phi' = \ell^{-1} F \circ \phi \circ h$ is an equivalent level structure, which satisfies $e_M(\phi'(1,0), \phi'(0,1)) = e_M(\ell^{-1}P^{(\ell)}, Q^{(\ell)}) = \zeta$. 

    Then $f_\zeta(E^{(\ell)}, [\phi']) = (E'', [\phi'_{\alpha'( Q)}]),$ where $\alpha': E^{(\ell)} \to E''$ has kernel generated by $\phi'(1,0) = \ell^{-1}P^{(\ell)}$, which is $F (\ker(\alpha))$. Thus we  have $F\alpha = \alpha' F$ and $E''$ is isomorphic to $E'^{(\ell)}$.  Moreover, $\phi'_{\alpha'( Q^{(\ell)})}(1,0) = \alpha'( Q^{(\ell)}) = \alpha' F(Q) =  F \alpha(Q)$. Since $\phi'_{\alpha'(Q^{\ell})}(1,0) = F\circ \phi(1,0)$, they are equivalent as $H'$-level structures.

    Therefore, $f_{\zeta}^{-1} F f_{\zeta} = \langle\ell\rangle^{-1} F$, so $F$ and $\langle \ell\rangle F^{\vee}$ are adjoints under the modified Weil pairing and thus have the same characteristic polynomial (over $\mathbb{T}$) and therefore the same trace. Using the Eichler-Shimura relation again, we compute:
    \[2 T_\ell =\Tr^{V_q}_{\mathbb{T}}(T_{\ell}) = \Tr^{V_q}_{\mathbb{T}}(F) + \Tr^{V_q^{*}}_{\mathbb{T}}(\langle \ell \rangle F^{\vee}) = 2 \Tr^{V_q}_{\mathbb{T}}(F),\] so $T_{\ell} = \Tr(\Frob)$.


    Thus, the characteristic polynomial of $F$ on $V_q$ is $x^2 - T_{\ell}x + \langle \ell\rangle \ell$, so we have \[P_1(X_{H',\FF_{\ell}}, u) = \det(1 - T_{\ell} u + \langle \ell\rangle \ell u^2 | \Pic^0(X_{H',\FF_{\ell}})).\]

    Now we relate these actions to $S_2(\Gamma_H)$. The actions of $T_{\ell}$, $\langle \ell\rangle$, $\Frob$, and $F^{\vee}$ on $\Pic^0(X_{H'})$ induce actions on the cotangent space $\Cot(X_{H'})$, which can be identified with $S_2(\Gamma_{H'})$. By \cite[(5.4.4) and (5.4.5)]{Codogni-Lido}, the action of $T_{\ell}$ on $\Pic^0(X_{H'})$ corresponds to the double coset operator $\tilde{T}_{\ell}$. We also have that $\langle \ell\rangle$ on $\Pic^0(X_{H'})$ corresponds to $\langle \tilde{\ell} \rangle$. 
    
    By the Eichler-Shimura isomorphism (see the appendix by Conrad in~\cite[Chapter 5]{RibetSteinPCMI} or Chapter 6 of~\cite{HidaElementary}), we have the isomorphism \[H^1(X_{H'}, \CC) \cong S_2(\Gamma_{H'}) \oplus \overline{S_2(\Gamma_{H'})}.\] Since $S_{2}(\Gamma_{H'}) \cong \overline{S_2(\Gamma_{H'})^{\vee}}$ as $T_{\ell}$-modules, under the isomorphism induced by $f_{\zeta}$, we can rewrite this as $S_2(\Gamma_{H'}) \oplus S_2(\Gamma_{H'})^{\vee}$. (See~\cite[Chapter 5, Theorem 5.1]{RibetSteinPCMI}): the action of $f_{e^{2\pi i/M}}$ corresponds to the action of $w_M \oplus \overline{w_M}$, where $w_M$ is the double coset operator corresponding to $\begin{pmatrix}
        0 & -1\\
        M & 0
    \end{pmatrix}$. The isomorphism $S_2(\Gamma_{H'}) \to \overline{S_2(\Gamma_{H'})}^{\vee}$, induced by $w_M$, is described in \cite[page 241]{Diamond-Shurman}.) Under this isomorphism, the action of $F$ on $\Pic^0(X_{H',\FF_{\ell}})$ corresponds to the action of $F \oplus \langle \ell\rangle F^{\vee}$ on $S_2(\Gamma_H) \oplus S_2(\Gamma_H)^{\vee}$. 
    
    Therefore, \[P_1(X_{H',\FF_{\ell}}, u)= \det(1 - \tilde{T}_{\ell} u + \ell \langle \tilde{\ell}\rangle u^2|S_2(\Gamma_H)).\]

\end{proof}

Following \cite{Codogni-Lido,Ribet1990}, we define $p$-old and $p$-new spaces of $S_2(\Gamma(N) \cap \Gamma_0(p))$ using the following inclusions of $S_2(\Gamma(N))$ into $S_2(\Gamma(N) \cap \Gamma_0(p))$: 
\begin{align*}
    i_1: S_2(\Gamma(N)) \rightarrow S_2(\Gamma(N)\cap \Gamma_0(p)) & & f(z)\mapsto f(z)\\
    i_2: S_2(\Gamma(N)) \rightarrow S_2(\Gamma(N)\cap \Gamma_0(p)) & & f(z)\mapsto f(pz)
\end{align*} 


By considering $H$-invariant subspaces, this induces the following inclusion maps:
\begin{align*}
i_3: S_2(\Gamma_H) \rightarrow S_2(\Gamma_{H_p}) & & f(z)\mapsto f(z)\\
    i_4: S_2(\Gamma_H) \rightarrow S_2(\Gamma_{H_p}) & & f(z)\mapsto f(pz)
\end{align*}

\begin{definition}\label{def:psubspaces}
Define the \textbf{$p$-old subspaces} as \begin{align*}
    S_2^{p-old}(\Gamma(N) \cap \Gamma_0(p)) &:= i_1(S_2(\Gamma(N))) \oplus i_2(S_2(\Gamma(N))) \\
    S_2^{p-old}(\Gamma_{H_p}) &:= i_3(S_2(\Gamma_H)) \oplus i_4(S_2(\Gamma_H)).
\end{align*}
The \textbf{$p$-new subspaces} are the orthogonal complements with respect to the Petersson inner product, i.e. \begin{align*}
    S_2^{p-new}(\Gamma(N) \cap \Gamma_0(p)) := \Big(S_2^{p-old}(\Gamma(N) \cap \Gamma_0(p))\Big)^{\perp}, &&
    S_2^{p-new}(\Gamma_{H_p}) := \Big(S_2^{p-old}(\Gamma_{H_p} \Big)^{\perp}
\end{align*}
\end{definition}

\cite[Theorem 5.5.2]{Codogni-Lido} gives an isomorphism between the subspace of $H$-invariant cuspforms in $S_2^{p-new}(\Gamma_0(p) \cap \Gamma(N))$ and $\ker(w_*)$. For our setup, it is more convenient to rephrase in terms of $S_2(\Gamma_{H_p})$.

\begin{lemma}\label{lem:pnewrephrase} Suppose $\det(H) = (\ZZ/N\ZZ)^{\times}$, and consider the action of $\GL_2(\ZZ/N\ZZ)$ on $S_2(\Gamma_0(p) \cap \Gamma(N)) \otimes \CC^{(\ZZ/N\ZZ)^\times}$ defined in \cite[Section 5 (5.5.1), Remark 5.5.3]{Codogni-Lido}, i.e. \begin{align*}
    \begin{pmatrix}
        d & 0\\
        0 & 1
    \end{pmatrix} \cdot (f_a)_a = (f_{ad})_{a} & & h \cdot (f_a)_a = (f_a[\tilde{h}_a]_2)_a, \det(h) = 1
\end{align*} where $\tilde{h}_{a} \in B_0(p)$ such that $\tilde{h}_a \equiv \Big( \begin{pmatrix}
    a & 0\\ 0 & 1
\end{pmatrix} h \begin{pmatrix}
    a & 0\\ 0 & 1
\end{pmatrix}^{-1}\Big)^T \pmod{N}$. Let $S^H$ denote the $H$-invariant subspace of cusp forms in $S$.

Then $S_2(\Gamma(N))^H = S_2(\Gamma_H)$,
$S_2(\Gamma_0(p) \cap \Gamma(N))^H = S_2(\Gamma_{H_p})$, and
    $S_2^{p-new}(\Gamma_{H_p}) = \Big(S_2^{p-new}(\Gamma_0(p) \cap \Gamma(N))\otimes \CC^{(\ZZ/N\ZZ)^\times}\Big)^H$.
\end{lemma}

\begin{proof} 
    The first two equalities follow from \cite[Lemma 6.5]{zywina2021computingactionscuspforms}.
    
    Now, we show that $S_2^{p-new}(\Gamma_{H_p})=\Big(S_2^{p-new}(\Gamma_0(p) \cap \Gamma(N))\otimes \CC^{(\ZZ/N\ZZ)^\times}\Big)^H$, i.e. that taking the p-new subspace commutes with taking the $H$-invariant subspace. By the action defined above, the matrices with determinant not equal to 1 act by shifting the indices and surjective determinant implies that the $H$-invariant subspaces consists of tuples of the form $(f_a)_a = (f)_a$ for some chosen representative $f$. So we only have to show $(S_2^{p-new}(\Gamma_{H_p}))=\Big(S_2^{p-new}(\Gamma_0(p) \cap \Gamma(N))\Big)^H$, i.e., the action of $h\in H$ respects the Petersson inner product $\langle \cdot,\cdot\rangle$ on $S_2(\Gamma_0(p) \cap \Gamma(N))$. 
    
    It suffices to consider $\det(h)=1$
    and to show that for each $a \in (\ZZ/N\ZZ)^\times$, $\langle f[\tilde{h}_a]_2, g[\tilde{h}_a]_2\rangle = \langle f, g\rangle.$ Note that $\tilde{h}_{a}^{-1} (\Gamma_0(p) \cap \Gamma(N))\tilde{h}_{a} \subset \Gamma_0(p) \cap \Gamma(N) \subset \SL_2(\ZZ)$, so by \cite[Proposition 5.5.2(a), Exercise 5.4.3]{Diamond-Shurman}, \[\langle f[\tilde{h}_a]_2, g[\tilde{h}_a]_2 \rangle_{ \Gamma_0(p) \cap \Gamma(N)} =\langle f[\tilde{h}_a]_2, g[\tilde{h}_a]_2 \rangle_{\tilde{h}_{a}^{-1} (\Gamma_0(p) \cap \Gamma(N))\tilde{h}_{a}} = \langle f, g \rangle_{\Gamma_0(p) \cap \Gamma(N)}.\] It follows that an orthogonal decomposition of $H$-invariant $f \in S_2(\Gamma_0(p) \cap \Gamma(N))$ into $p$-old and $p$-new components is an orthogonal decomposition into $H$-invariant $p$-old and $p$-new components. \end{proof}

Now, we relate the Ihara zeta function of $G(p, \ell, H)$ to the Hasse-Weil zeta functions of $X_{H,\FF_{\ell}}$ and $X_{H_p,\FF_{\ell}}.$

\begin{theorem}\label{thm:IharaZetaModularCurveProduct}
Let $H \subseteq \GL_2(\ZZ/N\ZZ)$ be a subgroup satisfying $B_1(N)\subseteq H \subseteq B_{0}(N)$ and let $G = G(p,\ell,H)$.
Denote by $X_{H,\FF_{\ell}}$ and $X_{H_p,\FF_\ell}$  the associated modular curves over $\FF_{\ell}$.
    Then, in the notation of Theorem \ref{thm:IharaZetaFormula-nonorientable} 
    \[\frac{Z(X_{H_p,\FF_\ell}, u)}{Z(X_{H,\FF_{\ell}}, u)^{2}}\zeta_{G(p,\ell,H)}(u)  = (1 - u^2)^{C_1(L)}(1 + u)^{-C_1(J)} \prod_{k > 1} (1 - (-1)^ku^{2k})^{C_k(L)}(1 - u^k)^{-C_k(J)}.\] 
\end{theorem}

\begin{proof}
    
    
    Following Definition \ref{def:psubspaces} and Lemma \ref{lem:pnewrephrase}, we decompose the space of cusp forms of $H_p$ into $p$-old and $p$-new subspaces: $$S_2(\Gamma_{H_p}) = S_2^{p-old}(\Gamma_{H_p}) \oplus S_2^{p-new}(\Gamma_{H_p}).$$ The $p$-old and $p$-new spaces are stable under $\tilde{T}_{\ell}$ following the arguments in \cite[Propositions 5.5.2, 5.6.2]{Diamond-Shurman}. We can write: \[Z(X_{H_p,\FF_\ell}, u) = \frac{\det(1 - \tilde{T_{\ell}} u + \ell \langle \tilde{\ell} \rangle u^2 | S_2^{p-old}(\Gamma_{H_p}))\det(1 - \tilde{T_{\ell}} u + \ell \langle \tilde{\ell} \rangle u^2 | S_2^{p-new}(\Gamma_{H_p}))}{(1-u)(1-\ell u)}.\]

     The $p$-old subspace is isomorphic to two copies of $S_2(\Gamma_H)$ as $\tilde{T}_{\ell}$-modules. We have \[Z(X_{H,\FF_{\ell}}, u)^2 = \frac{\det(1 - \tilde{T_{\ell}} u + \ell \langle \tilde{\ell}\rangle u^2 | S_2^{p-old}(\Gamma_{H_p}))}{(1-u)^2(1 - \ell u)^2},\] and therefore \[Z(X_{H_p,\FF_\ell}, u)Z(X_{H,\FF_{\ell}}, u)^{-2} = (1 - u)(1 - \ell u) \det(1 - \tilde{T_{\ell}} u + \ell \langle \tilde{\ell} \rangle u^2 | S_2^{p-new}(\Gamma_{H_p})).\]


     Now we relate this expression to the Ihara zeta function of the graph $G$. Since $H$ has surjective determinant, we have $|R_H| = 1$, $G$ is connected, and the order of $\ell$ in $R_{H}$ is $1$. 
     
     Combining Theorem \ref{thm:IharaZetaFormula-nonorientable} and Proposition \ref{prop:keromega}, 
     we have \[\zeta_{G}(u) = (1 - u)^{-1}(1 - \ell u)^{-1}\det(1 - Au + \ell Lu^2| \ker(\omega_{1, *}))^{-1} R(u),\] where $R(u)$ is the numerator of Theorem \ref{thm:IharaZetaFormula-nonorientable}.
        
     There is an isomorphism of $S_{2}^{p-new}(\Gamma_{H_p})$ with $\ker(\omega_{1,*})$, which simultaneously intertwines the action of $\tilde{T_{\ell}}$ and $\langle \tilde{\ell} \rangle$ on $S_{2}^{p-new}(\Gamma_{H_p})$ with, respectively, $A^*$ and $\langle \ell\rangle^{-1}$ on $\ker(\omega_{1, *})$ \cite[Theorem 5.5.2]{Codogni-Lido}. Here, $A^*$ is the adjoint of $A$ with respect to the inner product defined in Definition \ref{def:innerproduct} (also see \cite[Proposition 2.2]{Codogni-Lido}). Note that the  adjoint of $L$, whose action on $V(G)$ is exactly the action of $\langle \ell \rangle$ on $\Pic^0(X_{H
     })$, is $L^{-1}$, which acts as $\langle \ell \rangle^{-1}$. Taking adjoints, and noting that $A$ and $A^*$ have the same eigenvalues (\cite[Proposition 2.2.2]{Codogni-Lido}), we have $\det(1-uA + u^2\ell L|\ker(w))=\det(1 - u A^* + u^2 \ell L^*|\ker(w))=\det(1 - \tilde{T}_{\ell} u + \ell \langle \tilde{\ell}\rangle u^2|S_{2}^{p-new}(\Gamma_{H_p})).$  
     
     Therefore, \[Z(X_{H_p,\FF_\ell}, u) Z(X_{H,\FF_\ell}, u)^{-2} = (1-u)(1-\ell u)\det(1 - Au + QL u^2| \ker(\omega_{1, *})),\] and combining the two equations, we obtain \[\zeta_G(u) Z(X_{H_p,\FF_{\ell}}, u) Z(X_{H,\FF_{\ell}}, u)^{-2} =  R(u).\]\end{proof}

As a corollary, we obtain the case considered by \cite{Lei-Muller} and \cite{Sugiyama-zetas}.

\begin{corollary}\label{cor:BorelLevelStructureProdFormula}
        Let $G=(p,\ell,B_0(N))$. Let $X_0(pN)_{\FF_{\ell}}$ and $X_0(N)_{\FF_{\ell}}$ denote the modular curves over $\mathbb{F}_\ell$. 
    Then we have that 
    \[Z(X_0(pN)_{\FF_\ell}, u) Z(X_0(N)_{\FF_{\ell}}, u)^{-2} \zeta_{G}(u) = (1 + u)^{\chi(G^{-1})}(1 - u)^{\chi(G^{+1})},\] where $G^{+1}$, $G^{-1}$ are the orientable graphs associated to $G$ in Definition~\ref{def:OrientableGraphs}.
\end{corollary}

Corollary \ref{cor:BorelLevelStructureProdFormula} can be made  explicit by computing the Euler characteristics of $\chi(G(p,\ell,B_0(N))^{\pm 1})$ in elementary terms. We perform this computation in Theorem~\ref{thm:LevelStructureEulerCharacteristic-GeneralCase} in Appendix~\ref{app:ExplicitEulerChars}. Using this result, we give two examples below.

\begin{example}\label{ex:G(11,3)-product-formula}
    Let $p = 11$ and $\ell = 3$. We consider the Ihara zeta function of $G(p,\ell)$. Corollary \ref{cor:BorelLevelStructureProdFormula} relates $\zeta_{G(p,\ell)}(u)$ to the zeta functions $Z(X_0(11)_{\FF_3},u)$ and $Z(X_0(1)_{\FF_3}, u)$. The curve $X_0(11)$ is elliptic, given explicitly by $y^2 + y = x^3 - x^2 - 10x - 20$. The trace of Frobenius at $\ell = 3$ is $-1$, so we have that \[Z(X_0(11)_{\FF_3},u) = (1 + u + 3u^2)(1 - u)^{-1}(1 - 3u)^{-1}.\] The curve $X_0(1)$ has genus zero, so the Hasse-Weil zeta function is $Z(X_0(1)_{\FF_3},u) = ((1 - u)(1 - 3u))^{-1}$. The adjacency matrix for the graph $G(11, 3)$ is given by $A = \begin{pmatrix}
        1 & 3 \\
        2 & 2
    \end{pmatrix}$. Using Theorem \ref{thm:LevelStructureEulerCharacteristic-GeneralCase}, we commpute that $\chi(G(11,3)^{+1}) = 1$, and $\chi(G(11, 3)^{-1}) = -1$. Thus by Corollary \ref{cor:Ihara-determinant-J-involution}, we have 
    \begin{align*}
        \zeta_G(u) & = \frac{(1 - u)(1 + u)^{-1}}{\det(1 - uA + 3u^2)} \\
        & = \frac{1 - u}{(1 - u^2)(1 - 3u)(1 + u + 3u^2)}
    \end{align*}
    The product in Theorem \ref{thm:IharaZetaModularCurveProduct} is therefore given by 
    \[ \left(\frac{1 + u + 3 u^2}{(1 - u)(1 - 3u)}\right)((1- u)(1 - 3u))^2\left(\frac{1 - u}{(1 - u^2)(1 - 3u)(1 + u + 3u^2)}\right) 
    = \frac{1 - u}{1 + u},
    \] as predicted by the Theorem.
\end{example}

The next example demonstrates the case where $L$ is nontrivial. Note that we use notation from Appendix~\ref{app:ExplicitEulerChars} in this example.

\begin{example}
    Let $p = 11, \ell = 3$, and $H = B_1(5)$. We consider $G\coloneqq G(13,3, B_1(5))$. We have $(-3|13)=1$ and $p\equiv 1\pmod{12}$ so Theorem~\ref{thm:IharaFormulaPermutationCase} applies. In particular, the maps $J$ and $L$ are permutations. Because $2 =[\langle 3\Id\rangle : \langle 3\Id\rangle \cap \pm H]$, every cycle in $L$ has length $2$ and every cycle in $J$ has length $4$. There are $\#X=[\GL_2(\ZZ/5\ZZ):\pm B_1(5)]=12$ vertices and $4\cdot 12=48$ edges. Let $A$ be the adjacency matrix. By Theorem~\ref{thm:Ihara_Zeta_permutations_version}, we have \[
    \zeta_G(u) = \frac{(1-u^4)^{6}(1-u^4)^{-12}}{(1-u)(1-3u)f(u)},
    \]
    where 
    \begin{align*}
    f(u) &= (3 u^{2} + 2 u + 1)( 9 u^{4} - 6 u^{3} + 4 u^{2} - 2 u + 1)( 9 u^{4} + 4 u^{2} + 1) \\ 
    &\quad \cdot( 729 u^{12} - 486 u^{10} + 99 u^{8} - 8 u^{6} + 11 u^{4} - 6 u^{2} + 1)
    \end{align*}
    is the numerator of $Z(X_{B_1(5)\times B_0(13)}(65)_{\FF_3},u)$. We computed $\zeta_G$ with our Sagemath code available here~\cite{isogeny_graphs_code}, and we verified that $f(u)$ is the numerator of the zeta function of $X_{B_1(5)\times B_0(13)}(5)_{\FF_3}$ with Sutherland's Magma code~\cite{RSZB,AndrewVSutherland_Magma_2025}. Note that $X_1(5)$ has genus $0$ so its zeta function is $(1-u)^{-1}(1-3u)^{-1}$. The Taylor expansion of $\zeta_G$ at $0$ is 
    $\zeta_G(u) = 4u + 8u^2 + 40u^3 + 112u^4 + 184u^5 + \ldots$, so there are $4$ non-backtracking tailless cycles of length $1$, $8$ of length $2$, $40$ of length $3$, and so on. 
\end{example}

\begin{figure}[h]
    \centering
    \includegraphics[width=0.7\linewidth]{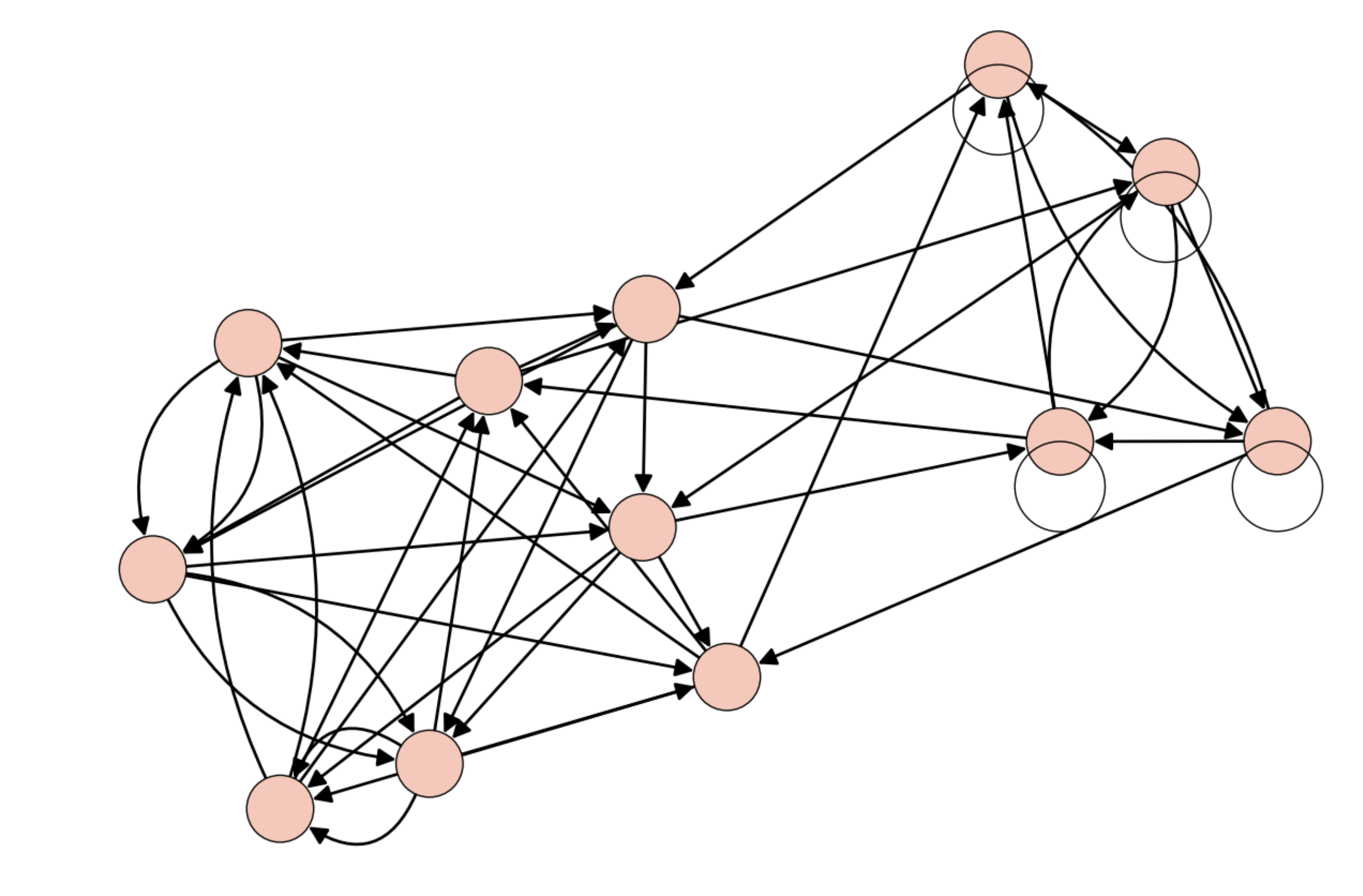}
    \caption{The  $3$-isogeny graph of supersingular elliptic curves with $B_1(5)$-level structure in characteristic 13.}
    \label{fig:G13_3_B15}
\end{figure}

\section{Applications}\label{sec:Applications}

In this section, we present several applications of the results in Sections \ref{sec:IharaDeterminantFormula} and \ref{sec:RelationToModularCurves}, as well as some explicit examples. First, we derive a point count formula for the number of points on the modular curve $X_0(p)$ over the finite field $\FF_{\ell^r}$ in terms of cycle counts in supersingular isogeny graphs. 
We then show that our zeta function counts ``isogeny cycles'' in the sense of \cite[Definition 3.1]{OrientationsAndCycles}. Using known results on the number of isogeny cycles, we give formulas for $\#X_0(p)_{\FF_{\ell^r}}$ in terms of imaginary quadratic class numbers. Finally, we extend an asymptotic for the number of cycles in $G(p,\ell)$ to $G(p,\ell, B_0(N))$, while also removing the assumption that $p \equiv 1 \pmod{12}$.   

\subsection{Point-counting on modular curves}\label{subsec:PointCountsonModularCurves}

Since the Hasse-Weil zeta functions in Theorem \ref{thm:IharaZetaModularCurveProduct} count points over extensions of $\FF_\ell$ on $X_0(pN)$ and $X_0(N)$, and the Ihara zeta function counts non-backtracking cycles, we can use the product formula of Theorem \ref{thm:IharaZetaModularCurveProduct} to obtain relationships between these counts. We make this explicit in the following proposition.

\begin{proposition}\label{prop:ModularCurvePointCount}
Let $G = G(p,\ell, B_0(N))$, and $N_r$ be the number of non-backtracking tailless cycles of length $r$ in $G$. Then we have that 
\[\#X_0(pN)(\FF_{\ell^r}) - 2 \# X_0(N)(\FF_{\ell^r}) + N_r = -\chi(G^{+1}) + (-1)^{r-1} \chi(G^{-1}).\]
\end{proposition}

\begin{proof}
From the definitions of the relevant zeta functions, we have that $\#X_0(pN)(\FF_{\ell^r})$ is the coefficient on $u^r$ in $u \frac{d}{du}Z(X_0(pN)_{\FF_{\ell^r}},u)$, and similar for $\#X_0(N)$, and $N_r$. We let $Z_1 = Z(X_0(pN)_{\FF_{\ell^r}},u)$, $Z_2 = Z(X_0(N)_{\FF_{\ell^r}},u)$, and $\zeta = \zeta_G(u).$ By Corollary \ref{cor:BorelLevelStructureProdFormula} of Theorem \ref{thm:IharaZetaModularCurveProduct} we have
\begin{align*}
u \frac{d}{du}\log(Z_1 Z_2 \zeta) & = u \frac{d}{du}\log((1 -u)^{\chi(G^{+1})}(1 + u)^{\chi(G^{-1})}) \\
u\frac{d}{du}\log(Z_1) + u \frac{d}{du}\log(Z_2) + u \frac{d}{du}\log(\zeta) & = \chi(G^{+1})u \frac{d}{du}\log(1 - u) \\&\quad+  \chi(G^{-1})u \frac{d}{du}\log(1 + u) \\
& = \chi(G^{+1})\frac{-u}{1 - u} + \chi(G^{-1})\frac{u}{1 + u} \\
& = -\chi(G^{+1})\sum_{n=0}^\infty u^{n+1} + \chi(G^{-1})\sum_{n=0}^\infty (-1)^n u^{n+1}. 
\end{align*}
The result follows by equating the coefficient on $u^r$ on both sides.
\end{proof}

As a corollary, we obtain a formula for the number of points on $X_0(p)$ over $\FF_{\ell^r}$ in terms of cycles in $G(p,\ell)$.

\begin{corollary}\label{cor:ModularCurvePointCountCor}
    Let $N_r$ be the number of non-backtracking tailless cycles of length $r$ in $G(p, \ell)$. Then we have
    \[\#X_0(p)(\FF_{\ell^r}) = 2(1 + \ell^r) - \chi(G^{+1}) + (-1)^{r-1}\chi(G^{-1}) - N_r.\]
\end{corollary}

\begin{proof}
Take $N = 1$ in Proposition \ref{prop:ModularCurvePointCount}, and use the fact that $X_0(1)$ has genus zero. 
\end{proof}

We can make the result of Corollary \ref{cor:ModularCurvePointCountCor} more explicit using counts of \emph{isogeny cycles} obtained by \cite{OrientationsAndCycles}.  An {\bf isogeny cycle} in $G(p, \ell)$ to be a closed walk, forgetting basepoint, such that no two consecutive edges compose to $[\ell]$, and that is not a power of another closed walk. The independence of basepoint means that isogeny cycles correspond to cycles up to cyclic permutation. They are non-backtracking in the sense used in this paper by Remark \ref{rmk:backtracking-reconcilliation}. Since isogeny cycles must be non-backtracking in every cyclic permutation, they must also be tailless. We therefore see that the isogeny cycles are exactly the primes in Definition \ref{def:abstract-iso-graph-defs}. It follows that $N_n = \sum_{r \mid n} r c_r$, where $c_r$ is the number of isogeny cycles of length $r$, and $N_n$ is the number of non-backtracking, tailless cycles of length $r$ in $G(p,\ell)$.

Isogeny cycles can be counted in terms of imaginary quadratic class numbers \cite{OrientationsAndCycles}. To recall this result, we begin with following definition:
\[
\mathcal{I}_r := 
\left\{
\text{imaginary quadratic orders } \mathcal{O} :
\begin{array}{l}
p \text{ does not split in the field containing } \mathcal{O} \\
p \text{ does not divide the conductor of } \mathcal{O} \\
\mathcal{O} \text{ is an } \ell \text{-fundamental order,} \\
(\ell) = \mathfrak{l} \overline{\mathfrak{l}} \text{ splits in } \mathcal{O}, \\
\text{and } [\mathfrak{l}] \text{ has order } r \text{ in } \operatorname{Cl}(\mathcal{O}).
\end{array}
\right\}
\]
Then we have
\begin{theorem}\cite[Corollary 7.3]{OrientationsAndCycles}\label{thm:OrientationsAndCyclesCount}
    Let $r > 2$ such that $\ell^r < p$, and $c_r$ denote the number of isogeny cycles in $G(p,\ell)$ of length $r$. Then
    \[rc_r = 2\sum_{\mathcal{O}\in \mathcal{I}_r} h(\mathcal{O}).\]
\end{theorem}

\begin{remark}\label{rmk:iso-cycle-weight-rmk}
In \cite{OrientationsAndCycles}, the authors go further by considering the case $\ell^r > p$. In this case, the quantity $r c_r$ can be written as a sum of the same class numbers, but with weights $w_\mathcal{O}$. These weights satisfy $1 \leq w_\mathcal{O} \leq 2$, but their exact description  would take us too far afield. Currently there is no explicit formula to describe the weights $w_\mathcal{O}$.
\end{remark}

Combining Theorem \ref{thm:OrientationsAndCyclesCount} with Corollary \ref{cor:ModularCurvePointCountCor} gives the following Theorem:

\begin{theorem}\label{thm:ModualrCurvePointCountExplicit}
Let $p, \ell$ be distinct primes and $r > 2$ such that $\ell^r < p$. Let $G \coloneqq G(p,\ell)$. Then we have that
\[\#X_0(p)(\FF_{\ell^r}) = 2(1 + \ell^r) - \chi(G^{+1}) + (-1)^{r - 1}\chi(G^{-1}) - 2\sum_{n \mid r} \sum_{\mathcal{O} \in \mathcal{I}_n} h(\mathcal{O}).\]
\end{theorem}

We give an example of Theorem \ref{thm:ModualrCurvePointCountExplicit} below.

\begin{example}\label{ex:ModularCurvePointCount}
    Let $p = 11$ and $\ell = 2$. We will will compute $\# X_0(p)_{\FF_{2^3}}$ using $G \coloneqq G(p, \ell)$. By Theorem~\ref{thm:LevelStructureEulerCharacteristic-GeneralCase}, we have that \[\chi(G^{+1}) = 1 \quad \text{and} \quad \chi(G^{-1}) = 0.\]

We then compute that $\mathcal{I}_3 = \{-31, -23\},$ and both of these orders have class number three. Using methods similar to the proof of Theorem \ref{thm:LevelStructureEulerCharacteristic-GeneralCase}, one can show that no loops in $G$ generate isogeny cycles of length one, so $c_1 = 0$. 

By Theorem \ref{thm:ModualrCurvePointCountExplicit}, we have 
    \begin{align*}
    \#X_0(11)(\FF_8) & = 2(1 + 8) - 1 + 0 - 2(0 + (3 + 3)) \\
    & = 18 - 1 - 12 \\
    & = 5.
    \end{align*}
    This result can be confirmed directly using the model given in Example \ref{ex:G(11,3)-product-formula}.
\end{example}

\begin{remark}\label{rmk:LM-miss-applications}
We reiterate that Theorem~\ref{thm:ModualrCurvePointCountExplicit}, as well as the applications presented in Section~\ref{subsec:applications-asymptotics} cannot be obtained from the results of Lei-M\"uller, because of the subtleties in their definition of the Ihara zeta function discussed in Section~\ref{sec:Background}. Thus these applications demonstrate the value of considering $G(p,\ell, H)$ as an abstract isogeny graph.  
\end{remark}

\subsection{Asymptotics for cycle counts in isogeny graphs}\label{subsec:applications-asymptotics}

We now turn our attention to extending the following theorem:

\begin{theorem}\cite[Theorem 7.1]{OrientationsAndCycles}\label{thm:OrientationsAndCycles7.1}
Let $p \equiv 1 \pmod{12}$. The number of non-backtracking closed walks of length $r$ in $G(p,\ell)$, taken up to order of traversal and starting point, asymptotically approaches $\ell^r / 2r$ as $r \to \infty$. Thus, the number of isogeny cycles in $G(p,\ell)$ is also asymptotically $\ell^r / 2r$ as $r \to \infty$.
\end{theorem}

Using Proposition \ref{prop:ModularCurvePointCount}, we give a short proof of a generalization of Theorem \ref{thm:OrientationsAndCycles7.1} to isogeny graphs with $B_0(N)$-level structure, which also removes the assumption that $p \equiv 1 \pmod{12}$.

\begin{theorem}\label{thm:CycleAsymptoticsThm}
Let $p, \ell$ be distinct primes and $N$ be coprime to $p\ell$. Let $G \coloneqq G(p,\ell, B_0(N))$, and $N_r$ be the number of non-backtracking cycles of length $r$ in $G$. Then $N_r$ asymptotically approaches $\ell^r$ as $r \to \infty$.
\end{theorem}

\begin{proof}
By Proposition \ref{prop:ModularCurvePointCount} we have that $N_r = \Theta(2\#X_0(N)(\FF_{\ell^r}) - \#X_0(pN)(\FF_{\ell^r}))$. We now note that by the Hasse-Weil bound, $|\#X - (1 + \ell^r)| \leq K \sqrt{\ell^r}$, where $X$ can be either $X_0(N)$ or $X_0(pN)$, and $K$ is a constant that depends on the curve, but not on $r$. Thus, as $r \to \infty$, $\# X / \ell^r \to 1$, for either $X = X_0(N)$ or $X = X_0(pN)$. Using Corollary~\ref{cor:ModularCurvePointCountCor}, we get that $N_r/\ell^r \to 1$ as $r \to \infty$.
\end{proof}

\begin{remark}
To recover Theorem \ref{thm:OrientationsAndCycles7.1} from Theorem \ref{thm:CycleAsymptoticsThm}, take $N = 1$, and note that $N_r$ counts all non-backtracking cycles of length $r$, so the number of non-backtracking cycles up to basepoint and direction of traversal is asymptotically $\ell^r/2r$.    
\end{remark}

\bibliographystyle{alpha}
\bibliography{bibliography.bib}

\appendix
\section{Explicit formulas for $\chi(G(p,\ell,H)^{\pm 1})$}\label{app:ExplicitEulerChars}   

In this appendix, we give formulas for $\chi((G(p,\ell, B_0(N))^{\pm1})$ in explicit, elementary terms for arbitrary $p>3$ and $\ell$. We also calculate the zeta function for more general $H$ under the hypothesis $p\equiv 1\pmod{12}$ and $(-\ell | p)=1$. 

\subsection{The Euler characteristics of the Borel-level structure isogeny graph}
We compute $\chi(G(p,\ell,B_0(N))$ in this subsection. Combining these formulas with Corollary~\ref{cor:Ihara-determinant-J-involution} allows one to calculate the zeta function of $G(p,\ell,B_0(N))$. 

\begin{theorem}\label{thm:LevelStructureEulerCharacteristic-GeneralCase}
    Let $p>3$ and $\ell$ be distinct primes and let $N$ be a positive integer coprime to $p\ell$. Let $(a|n)$ denote the Kronecker symbol and define 
    \[
    \epsilon_2(p,N)\coloneqq \begin{cases}
        \left(1-\left(\frac{-4}{p}\right)\right)\prod_{q|N} \left(1+\left(\frac{-4}{q}\right)\right) &: 4 \nmid N \\ 
        0 &: 4|N
    \end{cases}
    \]
    \[
    \epsilon_3(p,N)\coloneqq \begin{cases}
       \left(1-\left(\frac{-3}{p}\right)\right) \prod_{q|N} \left(1+\left(\frac{-3}{q}\right)\right) &: 9 \nmid N \\ 
        0 &: 9|N
    \end{cases}.
    \]
    \[
    \gamma(p,\ell,N)\coloneqq \left(1-\legendre{-\ell}{p}\right)\prod_{\substack{q|N \\ q\text{ odd }}}1+\legendre{-\ell}{q}.
    \]
    \[
       \delta_4(\ell)\coloneqq \begin{cases}
        \frac{1}{2}\left(\ell-\legendre{-1}{\ell}\right) &: \ell\not=2 \\ 
        1 &: \ell=2
    \end{cases} 
    \]
    \[
    \delta_3(\ell)\coloneqq  \begin{cases}
        \frac{2}{3}\left(\ell-\legendre{-3}{\ell}\right)) &: \ell\not=2 \\ 
        2 &: \ell=2
    \end{cases} \]
    \[
    \nu_{4\ell}(N)
    =\begin{cases}
        1 &: v_2(N) = 0 \\
        1 &: v_2(N)=1,\quad\ell\equiv 1\pmod{4} \\
        2 &: v_2(N) = 1,\quad \ell \equiv 3\pmod{4} \\
        0 &: v_2(N) >1,\quad \ell\equiv 1\pmod{4} \\
        2 &: v_2(N)=2,\quad \ell \equiv 3\pmod{8} \\ 
        4 &: v_2(N)\geq 2, \quad\ell\equiv 7\pmod{8} \\
        0 &: v_2(N)>2, \quad\ell\equiv 3\pmod{8}
    \end{cases}, \quad \nu_{\ell}(N) =\begin{cases}
        1 &: v_2(N)=0 \\ 
        0 &: v_2(N)>0,\quad\ell \equiv 3\pmod{8} \\ 
        2 &: v_2(N)>0,\quad\ell \equiv 7 \pmod{8}
    \end{cases}.\]
    \[
    r\coloneqq    \begin{cases}
        \frac{\epsilon_2(p,N)}{2}+\frac{1}{2}\left(1-\legendre{-2}{p}\right)\prod_{q|N}\left(1+\legendre{-2}{q}\right) &: \ell=2 \\
        \frac{1}{2}\#\Cl(\QQ(\sqrt{-\ell}))(\nu_{\ell}(N)+\nu_{4\ell}(N)) \gamma(p,\ell,N) &:\ell = 3 \\
        \frac{1}{2}\#\Cl(\QQ(\sqrt{-\ell})) \gamma(p,\ell,N)\nu_{4\ell}(N) 
        
        &: \ell\equiv 1 \pmod{4} \\
        \frac{1}{2}\#\Cl(\QQ(\sqrt{-\ell}))(\nu_{\ell}(N)+3\nu_{4\ell}(N)) \gamma(p,\ell,N)
        &:\ell \equiv 3 \pmod{8} \\
        \frac{1}{2}\#\Cl(\QQ(\sqrt{-\ell}))(\nu_{\ell}(N)+\nu_{4\ell}(N))\gamma(p,\ell,N)
        &:\ell \equiv 7 \pmod{8} \\
        
    \end{cases}
    \]
    Then
        \begin{align*}
    \chi(G(p,\ell, B_0(N))^{+1}) 
    &= \frac{(1-\ell)(p-1)\psi(N)}{24} + \left(\frac{1-\ell+2\delta_4(p,\ell)}{8}\right)\epsilon_2(p,N) \\ 
    &\quad + \left(\frac{2-2\ell+3\delta_3(p,\ell)}{12}\right) \epsilon_3(p,N) + \frac{r}{2}
   \end{align*}
    and $\chi(G(p,\ell,B_0(N))^{-1})=\chi(G(p,N,\ell)^{+1})-r$. 
\end{theorem}
\begin{proof}
    Let $\Gamma=G(p,\ell,B_0(N))$. The operator $L$ is the identity, so $X^{\pm1}=X$ and the number of vertices in either of $\Gamma^{\pm1}$ is the number of isomorphism classes of supersingular elliptic curves with $B_0(N)$-level structure. Such an isomorphism class may be identified with an isomorphism class of a pair $(E,C)$ where $C\leq E[N]$ is cyclic of order $N$, and two pairs $(E,C)$ and $(E',C')$ are isomorphic if there is an isomorphism of elliptic  curves  $u\colon E\to E'$ such that $u(C)=C'$. Let $\mathcal{Y}\coloneqq \bigcup_{i,j}\Aut(x_j)\backslash \widetilde{Y}_{ij}/\Aut(x_i)$ for $x_i=(E_i,C_i)$ ranging over the vertices in $X$. We have $Y^{+1}=\mathcal Y - \{O\in \mathcal{Y}:\widehat{O}=O\}$. Let $r=\#\{O\in \mathcal{Y} : \widehat{O}=O\}$. Then $\chi(G^{+1})=\#X - (\#\mathcal{Y}-r)/2$, so we compute $\#X$, $\#\mathcal{Y}$, and $r$. 
    
    The isomorphism classes of supersingular elliptic curves with $B_0(N)$-level structure are in bijection with the left ideal class set of an Eichler order $O$ of level $N$ in $\End^0(E)\simeq B_{p,\infty}$, the quaternion algebra ramified at $\{p,\infty\}$; see, for example,~\cite[Remark 42.3.10]{Voight}. Thus by Theorem~\cite[30.1.5]{Voight} we have
    \[
    \#X =  \frac{(p-1)\psi(N)}{12}+\frac{\epsilon_2(p,N)}{4}+\frac{\epsilon_3(p,N)}{3}. 
    \]
    
 We now count the number of edges of $\Gamma^{+1}$.  First we compute $\#\mathcal Y$.  
    Let $Y_i\coloneqq \{y\in Y: s(y)=x_i\}$. We have an action of $\Aut(x_i)$ on $Y_i$. To count $\#\mathcal{Y}$ we count orbits of this action as $x_i$ ranges over $X$. This is $\ell+1$ if $\Aut(x_i)=\{[\pm1]\}$. Let $x_i=(E_i,C_i)$ and suppose $\Aut(x_i)\simeq C_4$ is generated by $u$.  Since $u^2=[-1]$, the characteristic polynomial of $\restr{u}{E_i[\ell]}$ is $x^2+1\Mod{\ell}$. The number of fixed edges is equal to the number of roots of $x^2+1$ modulo $\ell$ (this is clear if $x^2+1\Mod{\ell}$ has $0$ or $2$ roots and it has one root only for $\ell=2$; note that for $\ell=2$ the automorphism $u$ acts as $\begin{pmatrix}
        1&1 \\ 0 & 1
    \end{pmatrix}$ so it only fixes one edge). If $\Aut(x_i)\simeq C_6$ then it is generated by an automorphism $u$ of order $6$ that satisfies $x^2-x+1$ and again the number of fixed edges is the number of roots of $x^2-x+1$ modulo $\ell$. This is $1+\left(-3|\ell\right)$ if $\ell$ is odd and $0$ if $\ell=2$. Thus if $\Aut(x_i)\simeq C_4$ we have 
    \[
    \sum_j\#\Aut(x_j)\backslash \widetilde{Y}_{ij}/\Aut(x_i) = \begin{cases}
        \frac{\ell+1-(1+(-1|\ell))}{2}+1+\legendre{-1}{\ell} &: \ell\not=2 \\ 
        2 &: \ell=2
    \end{cases} 
    \]
    and if $\Aut(x_i)\simeq C_6$ then 
    \[
    \sum_j\#\Aut(x_j)\backslash \widetilde{Y}_{ij}/\Aut(x_i) = \begin{cases}
        \frac{\ell+1-(1+\left(-3|\ell\right))}{3}+1+\legendre{-3}{\ell} &: \ell\not=2 \\ 
        1 &: \ell=2
    \end{cases}.
    \]
     For $x=(E,C)$ let $w_{x}\coloneqq\#\Aut(x)/2$. The number of orbits of the action of $\Aut(x_i)$ on  $L_i$  is $(\ell+1+(w_{x_i}-1)(1-\left(-d|p\right)))/w_{x_i}$, so $\delta_4(p,\ell)$ is the difference between $\ell+1$ and the number of orbits when $\Aut(x_i)\simeq C_4$, and similarly for $\delta_3(p,\ell)$. Now we calculate the number of isomorphism classes of $(E_i,C_i)$ with automorphism group $C_4$ and $C_6$. These counts are again given by global embedding numbers. The number of isomorphism classes $(E,C)$ with $\Aut(E,C)\simeq C_4$ is equal to the number of conjugate pairs of embeddings of $\ZZ[i]$ into Eichler orders of level $N$ in $B_{p,\infty}$, which is given by $\epsilon_2(p,N)/2$; see~\cite[Example 30.7.4]{Voight}. Similarly the number of isomorphism classes of pairs $(E,C)$ with $\Aut(E,C)\simeq C_6$ is $\epsilon_3(p,N)/2$. Therefore 
    \[
    \#\mathcal{Y} = (\ell+1)\cdot \#X - \delta_4(p,\ell)\frac{\epsilon_2(p,N)}{2} - \delta_3(p,\ell)\frac{\epsilon_3(p,N)}{2}.
    \]

    We now compute $\#\{O:\widehat{O}=O\}$. Once again, we do this by counting embedding numbers of certain quadratic orders into Eichler orders of level $N$. We claim that \[
    \#\{O:\widehat{O} = O\} = \frac{1}{2}\sum_{\substack{S \\ (\disc S) | 4\ell}}\sum_{[I]\in cls\OO}m(S,\OO_R(I);\OO_R(I)^{\times}),
    \] 
    where the sum is over quadratic orders $S$ of discriminant dividing $4\ell$, namely $\ZZ[i]$, $\ZZ[\sqrt{-\ell}]$, and when $\ell\equiv3\pmod{4}$, $\ZZ[(1+\sqrt{-\ell})/2]$. We will show each self dual orbit corresponds to a complex conjugate pair of embeddings of such an order. 
    To this end, we split into cases according to whether or not a self-dual orbit $O$ contains an endomorphism of trace zero. 
    We claim that if $\widehat{O}=O$ and there is no element of trace zero in $O$ then $O=\Aut(x)(1+g(i))\Aut(x)$ where $g\colon \ZZ[i]\hookrightarrow\End(x)$ is an embedding. Suppose $\widehat{O}=O$ and that $O$ has no element of trace zero. There exists $\alpha\in O$ and an automorphism $u\in \Aut(x)$ such that $\widehat{\alpha}=u\alpha$. 

     Since $u$ is an automorphism we have $u^{-1}=\widehat{u}$, so $u$ and $\alpha$ commute: 
    \[
    \widehat{\alpha} = u\alpha \implies \alpha =\widehat{\alpha}u^{-1} =u\alpha u^{-1}.\]
     Thus either $u=-1$ or $u\not\in \ZZ$ and $\alpha \in \ZZ[u]$. 
If $u=[-1]$, then $\alpha$ satisfies $X^2+\ell$ and thus has trace $0$, a contradiction. 

     If $u\not=[-1]$ then $j(x)$ is either $1728$ or $0$ and $\alpha\in\ZZ[u]$. Suppose $j(E)=1728$. Then $\widehat{u}=-u$ since $u^4=1$ but $u\not=\pm 1$. A calculation shows $\widehat{\alpha}=u\circ \alpha$ implies $\ell=2$ and $\alpha=\pm1\pm u$. 
  
    Thus $\Aut(x) \alpha \Aut(x) = \{\pm (1\pm g(i))\}$ where $g\colon \ZZ[i]\hookrightarrow \End(x)$ maps $i$ to $u$. 
    In particular, $\Aut(x)\alpha\Aut(x)$ contains elements corresponding to the unique pair of conjugate embeddings of $\ZZ[i]$ in $\End(x)$ and no elements corresponding to an embedding of an order containing $\ZZ[\sqrt{-\ell}]$. Another calculation shows that if $u$ is a primitive third or sixth root of unity then we cannot have $\widehat{\alpha}=u\alpha$ for an endomorphism $\alpha$ of norm $\ell$.

    We claim that an orbit $O$ satisfying $\widehat{O}=O$ that contains an element of trace $0$ corresponds to a unique pair of $\Aut(E)$-conjugacy classes of embeddings $f,\widehat{f}\colon \ZZ[\sqrt{-\ell}]\hookrightarrow \End(x)$. Suppose $\alpha,\beta\in O$ both have trace $0$ and norm $\ell$, and assume towards a contradiction that $\beta\not=\pm s\alpha s^{-1}$ for any $s\in \Aut(E)$. There exist $u,v\in \Aut(E)$ such that $\alpha= u\beta v$, so $v\alpha v^{-1} = vu\beta$. Then $s\coloneqq vu\not=\pm1$, so $s\alpha$ and $\alpha$ do not commute, since the only solutions to $X^2+\ell=0$ in $\QQ(\alpha)$ are $\pm \alpha$ and $vu\beta=v\alpha v^{-1}$ satisfies $X^2+\ell=0$. Thus $\End(x)$ contains the order generated by $s$ and $\alpha$, an order of discriminant $4\ell^2$ (when $j(x)=1728$ and $u^2+1=0$) or $9\ell^2$ (when $j(x)=0$ and $u^2\pm u+1=0$). This discriminant must be divisible by $p$. 
    But $p\not=2,3$ or $\ell$, a contradiction. 
    We conclude that 
    \[
    \#\{O:\widehat{O} = O\} = \frac{1}{2}\sum_{S: \disc S | 4\ell}\sum_{[I]\in cls\OO}m(S,\OO_R(I);\OO_R(I)^{\times}).
    \]
     For a quadratic order $S$, we have by~\cite[Theorem 30.7.3]{Voight} that
    \[
    \sum_{[I]\in Cls\OO} m(S,\OO_R(I);\OO_R(I)^{\times})=h(S)\prod_{q|Np} m(S,\OO_q,\OO^{\times}_q)
    \]
     so it remains to compute $m(S_q,\OO_q;\OO_q^{\times})$ for $q|Np$ and $S=\ZZ[\sqrt{-\ell}]$, for $S=\ZZ[(1+\sqrt{-\ell})/2]$ when $\ell\equiv 3\pmod{4}$, and finally $S=\ZZ[i]$ when $\ell=2$. For $q=p$ and each quadratic order $S$ above, the completion $S_p$ is integrally closed since $p\not=2$ and $\OO_p$ is maximal, so   $m(S_q,\OO_q;\OO_q^{\times})=(1-(-\ell|q))$ by~\cite[Proposition 30.5.3]{Voight}. It remains to compute $m(S_q,\OO_q;\OO_q^{\times})$ for $q|N$; for this we use~\cite[Proposition 30.6.12]{Voight}. There are many cases to consider. 
    
    First, assume $\ell=2$, so $Np$ is odd. For $q|N$ we have that $m(\ZZ_q[\sqrt{-2}],\OO_q;\OO_q^{\times})=1+(-2|q)$. 

Now suppose $\ell$ is odd and assume first $\ell\equiv 1\pmod{4}$. Then for $q|N$ odd we have $m(\ZZ_q[\sqrt{-\ell}],\OO_q;\OO_q^{\times})=1+\left(\frac{-\ell}{q}\right)$. If $2|N$ we have $m(\ZZ_2[\sqrt{-\ell}],\OO_2,\OO_2^{\times})=1$ if $v_2(N)=1$ and $0$ otherwise. 

    Now suppose $\ell\equiv 3\pmod{4}$. Now we must compute $m(S_q,\OO_q;\OO_q^{\times})$ for $S=\ZZ[\sqrt{-\ell}]$ and $S=\ZZ[(1+\sqrt{-\ell})/2]$. For $q$ odd, we have $\ZZ_q[\sqrt{-\ell}]=\ZZ_q[(1+\sqrt{-\ell})/2]$ and $m(S_q,\OO_q;\OO_q^{\times})=1+\left(\frac{-\ell}{q}\right)$ again. Suppose $q=2$ and let $e=v_2(N)$. We have 
    \[
    m(\ZZ_2[\sqrt{-\ell}],\OO_2;\OO_2^{\times})
    =\begin{cases}
        2 &: e = 1 \\
        2 &: e=2,\quad \ell \equiv 3\pmod{8} \\ 
        4 &: e\geq 2, \quad\ell\equiv 7\pmod{8} \\
        0 &: e>2, \quad\ell\equiv 3\pmod{8}
    \end{cases}.\]
    and
    \[
    m(\ZZ_2[(1+\sqrt{-\ell})/2],\OO_2;\OO_2^{\times}) = \begin{cases}
        0 &: \ell \equiv 3\pmod{8} \\ 
        2 &: \ell \equiv 7 \pmod{8}
    \end{cases}.
    \]

    For $\ell=2$ we also need to include $\sum_{[I]\in Cls\OO}m(\ZZ[\sqrt{-1}],\OO_R(I);\OO_R(I)^{\times})$; for this we compute $m(\ZZ_q[\sqrt{-1}],\OO_q;\OO_q^{\times})$ which is is $1+(-4|q)$ if $q|N$ is odd or if $q=2$ and $v_2(N)=1$, and $0$ if $q=2$ and $v_2(N)>1$ (see~\cite[Example 30.6.14]{Voight}). 

    Thus using~\cite[Theorem 7.24]{cox} to obtain $h(-3)=h(-12)$, $h(-4\ell)=3h(-\ell)$ for $\ell>3$ and $\ell\equiv 3\pmod{8}$, and $h(-4\ell)=h(-\ell)$ for $\ell\equiv 7\pmod{8}$, we get 
    \[
    \#\{\OO:\widehat{O}=O\} = \begin{cases}
        \frac{\epsilon_2(p,N)}{2}+\frac{1}{2}\left(1-\legendre{-2}{p}\right)\prod_{q|N}\left(1+\legendre{-2}{q}\right) &: \ell=2 \\
        \frac{1}{2}\#\Cl(\QQ(\sqrt{-\ell}))(\nu_{\ell}(N)+\nu_{4\ell}(N)) \gamma(p,\ell,N) &:\ell = 3\\
        \frac{1}{2}\#\Cl(\QQ(\sqrt{-\ell})) \gamma(p,\ell,N)\nu_{4\ell}(N)  
        &: \ell\equiv 1 \pmod{4} \\
        \frac{1}{2}\#\Cl(\QQ(\sqrt{-\ell}))(\nu_{\ell}(N)+3\nu_{4\ell}(N)) \gamma(p,\ell,N)
        &:\ell \equiv 3 \pmod{8} \\
        \frac{1}{2}\#\Cl(\QQ(\sqrt{-\ell}))(\nu_{\ell}(N)+\nu_{4\ell}(N))\gamma(p,\ell,N)
        &:\ell \equiv 7 \pmod{8}. \\
        
    \end{cases}
    \]

    We conclude
\begin{align*}
    \chi(G(p,N,\ell)^{+1}) &= \#X - \frac{1}{2}(\#\mathcal{Y} - r) \\ 
    &= \frac{1-\ell}{2}\#X + \delta_4(p,\ell)\frac{\epsilon_2(p,N)}{4}+\delta_3(p,\ell)\frac{\epsilon_3(p,N)}{4} +\frac{r}{2}\\
    &= \frac{(1-\ell)(p-1)\psi(N)}{24} + \left(\frac{(1-\ell)}{8}+\frac{\delta_4(p,\ell)}{4}\right)\epsilon_2(p,N) \\ 
    &\quad + \left(\frac{(1-\ell)}{6}+\frac{\delta_3(p,\ell)}{4}\right) \epsilon_3(p,N) + \frac{r}{2} \\
    &= \frac{(1-\ell)(p-1)\psi(N)}{24} + \left(\frac{1-\ell+2\delta_4(p,\ell)}{8}\right)\epsilon_2(p,N) \\ 
    &\quad + \left(\frac{2-2\ell+3\delta_3(p,\ell)}{12}\right) \epsilon_3(p,N) + \frac{r}{2}. 
\end{align*}
    Since $Y^{-1} = \mathcal{Y}\sqcup\{O\colon \widehat{O}=O\}$ we have $\#Y^{-1}=\#Y^{+1}+2r$ and \[\chi(G(p,N,\ell)^{-1})=\chi(G(p,N,\ell)^{+1})-r.\] 

\end{proof}

\subsection{The determinant formula when $L\not=1$}
An explicit form for the Ihara determinant formula for $G(p,\ell,H)$ for general $H\leq \GL_2(\ZZ/N\ZZ)$  involves computing $C_k(J)$ and $C_k(L)$. The function $L\colon X\to X$ is a permutation. To understand the cycle decomposition of $L$, we need more notation, following~\cite{Terao}. 
Given an isomorphism $\phi\colon (\ZZ/N\ZZ)^2\to E[N]$,  define 
\[
A_{E,\phi}\coloneqq \{\phi^{-1}\circ u \circ \phi : u\in \Aut(E)\}.
\]
If $h\in H$ then $A_{E,\phi\circ h}= h^{-1}A_{E,\phi}h$. By~\cite[Lemma 3.2]{Terao}, we have that $\langle \ell\rangle\cap A_{E,\phi}H$ is a subgroup of $\GL_2(\ZZ/N\ZZ)$.
 Define the {\bf index} of $\phi$ to be the least positive integer $m$ such that $\ell^m\in A_{E,\phi}H$. Then the index of $\phi$ is $[\langle \ell I\rangle:\langle \ell I \rangle \cap A_{E,\phi}H]$. If $h\in H$ and $\ell^m\in A_{E,\phi\circ h}H$ then there exists $h'\in H$ such that
 \[
 \ell^m = h^{-1}\phi^{-1}u\phi hh'
 \]
 so $\ell^m = \phi^{-1}u\phi (hh'h^{-1})\in A_{E,\phi}H$; thus $\ell^m\in A_{E,\phi\circ h}H$ implies  $\ell^m\in A_{E,\phi}H$. Repeating the argument but reversing the roles of $\phi$ and $\phi\circ h$ shows that if $\ell^m\in A_{E,\phi}H$ then  $\ell^m\in A_{E,\phi\circ h}H$. Thus $\phi$ and $\phi\circ h$ have the same index; define the {\bf index} of $[\phi]$ to be the index of any representative level structure $\phi$ for the $H$-level structure $[\phi]$. Denote the index of $[\phi]$ by $m_{[\phi]}$. Then 
 the length of the orbit of $x_i=(E_i,[\phi_i])$ is the index $m_{[\phi_i]}$ of $[\phi_i]$, so 
    \[
    C_k(L) = \frac{\#\{i: m_{[\phi_i]} = k\}}{k}. 
    \]

Computing  $C_k(J)$ and $C_k(L)$ for arbitrary $H\in \GL_2(\ZZ/N\ZZ)$ seems like a daunting task. We give a formula for $\zeta_{G(p,\ell,H)}$ when $p\equiv 1\pmod{12}$ and $(-\ell|p)=1$, so there are no curves with extra automorphisms and no self-dual loops in $G(p,\ell)$. 
\begin{theorem}\label{thm:Ihara_Zeta_permutations_version}
    Let $p$ and $\ell$ be distinct primes. Suppose $p\equiv 1 \pmod{12}$ and $(-\ell|p)=1$. 
    Let $H\subseteq \GL_2(\ZZ/N\ZZ)$ 
    and let $\Gamma=\Gamma(p,H,\ell)=(X,Y,s,t,J,L)$ with $L$ and $J$ defined as in Section~\ref{sec:level-structure}. 
    Then $L$ and $J$ are permutations whose cycle decomposition consists entirely of cycles of length $m(\ell,N)\coloneqq[\langle \ell I \rangle : \langle \ell I \rangle \cap \pm H]$ and $2m(\ell,N)$, respectively. The number of vertices of $\Gamma$ is $\#X = \frac{(p-1)[\GL_2(\ZZ/N\ZZ):\pm H]}{12}$ and the number of edges is $(\ell+1)\cdot \#X$. Consequently, 
    \[
    \zeta_{\Gamma}(u) = \frac{(1-(-1)^{m(\ell,N)}u^{2m(\ell,N))})^{\#X/m(\ell,N)}(1-u^{2m(\ell,N)})^{-(\ell+1)\cdot\#X/2m(\ell,N)}}{\det(1-uA+u^2LA)}.
    \]

\end{theorem}
\begin{proof}
    The number of vertices $\#X$ is the number of isomorphism classes of supersingular curves over $\overline{\FF_p}$ with $H$-level structure. Two supersingular elliptic curves with $H$-level structure $(E,[\phi])$ and $(E',[\phi'])$ are isomorphic if and only if $\phi'=u\circ \phi\circ h$ for some isomorphism $u\colon E\to E'$ and some $h\in H$. Since $p\equiv1\pmod{12}$ the automorphism of any supersingular elliptic curve over $\overline{\FF_p}$ is $\{[\pm1]\}$, so we may count $\#X$ by counting the number of possible $H$-level structures on a fixed supersingular elliptic curve $E$. We have that $(E,[\phi])\simeq (E,[\phi'])$ if and only if $\phi=\pm \phi' \circ h$ for some $h\in H$. If we fix one level-$N$ structure $\phi\colon (\ZZ/N\ZZ)^2\to E[N]$ then every other level-$N$ structure $\phi'$ is of the form $\phi'=\phi \circ g$ for some $g\in \GL_2(\ZZ/N\ZZ)$, and it will result in an isomorphic pair $(E,[\phi'])$ if and only if $g=\pm h$ for some $h\in H$. Thus the number of non-isomorphic pairs $(E,[\phi'])$ is equal to the index of $\pm H$ in $\GL_2(\ZZ/N\ZZ)$. 
    
    The order of the orbit of $(E,[\phi])$ under $L\colon X\to X$ is $m_{[\phi]} \coloneqq m=[\langle \ell I\rangle : \langle \ell I\rangle \cap \pm H]$ is constant, independent of $E$ or $\phi$. This gives the desired formula for $C_k(L)$. 
    
    Edges in $Y$ are sets $\{\pm \alpha\}$ where $\alpha$ is a morphism of curves with $H$-level structure, so $J$ is a permutation of $Y$.   Because we assume $(-\ell | p)=1$, the imaginary quadratic order $\ZZ[\sqrt{-\ell}]$ doesn't embed into the endomorphism ring of any supersingular elliptic curve in characteristic $p$. Thus there are no endomorphisms $\alpha$ of norm $\ell$ such that $\widehat{\alpha}=-\alpha$, so the orbit of an edge $y$ has order $2m$, where $m$ is as above. This yields the claimed formula for $C_k(J)$, since there are $\ell+1$ many edges at each vertex in $X$. 
\end{proof}

\section{The realization of an isogeny graph}\label{app:realization}
We now extend to abstract isogeny graphs the construction of Serre~\cite[\S 2.1, pg. 14]{trees} of a CW complex associated to a graph $\Gamma$. First, we define some relevant notions.

\begin{definition}\label{def:TreeDefs}
A {\bf tree} is a connected orientable graph with no cycles. A {\bf subtree} of a graph is any subgraph isomorphic to a tree. A {\bf spanning tree} in a graph $\Gamma$ is a subtree containing all the vertices of $\Gamma$. 
\end{definition}

Let $\Gamma=(X,Y,s,t,J,L)$ be an abstract isogeny graph. Let $\realization\Gamma$ be the quotient of $X\sqcup (Y\times[0,1])$ by the equivalence relation generated by the following relations:
\begin{itemize}
    \item $(y,r)\sim (Jy, 1-r)$,
    \item $(y,0) \sim s(y)$, and 
    \item $(y,1) \sim t(y)$,
\end{itemize}
and where the topology on $\realization\Gamma$ is the quotient topology (we give the usual topology on $[0,1]$ and the discrete topologies to $X$ and $Y$). 

Observe that 
\[
s(y) \sim (y,0) \sim (Jy, 1) \sim t(Jy),
\]
so in $\realization\Gamma$, two vertices $x,x' \in X$ are identified if there is an edge $y$ with $x=s(y)$ and $t(Jy)=x'$, just as in the construction of $\Gamma^{\pm1}$. Similarly we have 
\[
(y,r)\sim (Jy,1-r)\sim (J^2y,r)
\]
so edges in the same equivalence class under $\sim_Y$ will have the same image in $\realization\Gamma$. 

    The realization $\realization\Gamma$ can be given the structure of a CW-complex of dimension $\leq 1$. The $0$-cells are the image in $\realization\Gamma$ of the vertices $X$, together with a $0$-cell at the class of $(y,1/2)$ for each $y\in Y$ with $J[y]=[y]$. The $1$-cells correspond to images of subsets of $Y$ of the form $\{y,Jy\}$. 
 
\begin{proposition}\label{prop:realizationretract}
  The realization $\realization\Gamma^{+1}$ of $\Gamma^{+1}$ is a deformation retract of $\realization\Gamma$. 
\end{proposition}
\begin{proof}
    First, we define a map $f\colon \realization\Gamma^{+1}\to \realization\Gamma$. Given  $[([y],r)]\in \realization\Gamma^{+1}$, define $f([([y],r)])=[(y,r)]$. This map is well-defined: if $y'=J^2y$ then
    \[
    (y,r) \sim (Jy, 1-r) \sim (J^2y, r).
    \]
    This implies $f$ does not depend on the choice of representative $y\in [y]$. It is an injection, since if 
    \[
    [(y,r)]=[(y',r')] 
    \]
    then either $r=r'$ and $y\sim_Y y'$ so $[([y],r)]=[([y'],r')]$ or $r'=1-r$ and $Jy\sim_Y y'$.

    From now on, identify $\realization\Gamma^{+1}$ with its image in $\realization\Gamma$, i.e. the elements $[(y,r)]$ with $J[y]\not=[y]$. Consider a $1$-cell corresponding to $y\in Y$ with $J[y]=[y]$. We claim that this $1$-cell is homeomorphic to a segment $[0,1]$, where the boundary point of the $1$-cell corresponding to $0$ is in $\realization\Gamma^{+1}$ and the boundary point corresponding to $1$ is in the boundary of $\realization\Gamma$. The subspace is $\{[(y,r)]:r\in [0,1]\}$. Define the map $(y,r)\mapsto -|2r-1|+1$; this map is constant on equivalence classes, since 
    \[
    1-|2(1-r)-1| = 1-|1-2r| = 1-|2r-1|,
    \]
    and is a bijection, since 
    \[
    1-|2r-1|=1-|2r'-1| \implies 2r-1 = \pm (2r'-1) \implies r=r' \text{ or } r=1-r'.
    \]
    The closure of $\realization\Gamma-\realization\Gamma^{+1}$ consists of $1$-cells corresponding to $y\in Y$ with $[Jy]=[y]$. Each such 1-cell $y$ can be contracted to $s(y)\in \realization\Gamma^{+1}$, so $\realization\Gamma^{+1}$ is a deformation retract of $\realization\Gamma$. More precisely, we have a homotopy 
    \begin{align*}
        H\colon \realization\Gamma\times [0,1] &\to \realization\Gamma^{+1} \\
        ([(y,r)],t) &\mapsto \begin{cases}
            [(y,t(1-|2r-1|))]&\colon J[y]=[y] \\
            [(y,r)]&\colon J[y] \not= [y].
        \end{cases}
    \end{align*}
    
\end{proof}
Let $\Gamma=(X,Y,s,t,J,L)$ be an isogeny graph. We now assume $\Gamma$ is {\bf finite}, i.e. $\vertices\Gamma$ and $\edges\Gamma$ are finite sets.
For a finite, connected, orientable isogeny graph $\Gamma$, define 
\[
B_1(\Gamma)\coloneqq \frac{1}{2}\left(\#\edges\Gamma\right)-\#\vertices\Gamma +1.
\]
\begin{corollary}\label{cor:realization_homotopy_type}
      The realization $\realization\Gamma$ of a finite abstract isogeny graph $\Gamma$ has the homotopy type of a {\em bouquet of circles}, where the number of circles is $B_1(\Gamma^{+1})$.  
\end{corollary}

\begin{proof}
    This follows from~\cite[Corollary 1]{trees} and the previous proposition: $\realization\Gamma$ and $\realization\Gamma^{+1}$ are homotopy equivalent, and $\realization\Gamma^{+1}$, being the realization of an orientable graph, has the homotopy type of a bouquet of circles where the number of circles is $B_1(\Gamma^{+1}).$

\end{proof}
\begin{remark}
     We can combine the arguments in~\cite[Corollary 1]{trees} and in the proof of Proposition~\ref{prop:realizationretract}  above to see that $\realization\Gamma$ has the homotopy type of a bouquet of $B_1(\Gamma^{+1})$ circles. Let $T$ be any spanning tree in $\Gamma^{+1}$. Then its image in  $\realization \Gamma^{+1}\subseteq \realization \Gamma$ is contractible. After we contract the subspace corresponding to $T$ in $\realization\Gamma$ to a point, we are left with a ``bouquet of circles with thorns,'' that is, a wedge sum of circles and line segments. The circles correspond to edges in $\edges\Gamma^{+1}-\edges T$, i.e. 
    $\{[y],J[y]\} \subseteq Y^{+1} - T$ and the segments correspond to $[y]\in Y/{\sim_Y}-T$ such that $J[y]=[y]$. Thus there are $s=\#\{[y]:[Jy]=[y]\}$ thorns and 
    \[
    \frac{1}{2}(\#\edges\Gamma^{+1}-\#\edges T)
    \]
    circles. By ~\cite[Proposition 12]{trees}, we have $\#\edges T = 2(\#\vertices T -1)$, and $\#\vertices T = \#\vertices \Gamma^{+1}$. We may further contract each ``thorn'' to the point where the circles and thorns meet, so we're left with a bouquet of $B_1(\Gamma^{+1})$ circles, as in the statement of the proposition.
\end{remark}

\end{document}